\documentclass[11pt]{article}
\usepackage{amssymb,latexsym}
\usepackage{amsmath}
\usepackage{amsthm}
\usepackage{authblk}
\usepackage{graphicx}
\usepackage{xcolor}
\makeatletter

\newcommand{\Rmnum}[1]
{\expandafter@slowromancap\romannumeral#1@}
\makeatother

\numberwithin{equation}{section}
\newtheorem{theorem}{Theorem}[section]
\newtheorem{proposition}[theorem]{Proposition}
\newtheorem{lemma}[theorem]{Lemma}

\usepackage{multirow}

\newcommand{\R}{{\mathbb{R}}}
\theoremstyle{definition}

\newtheorem{remark}{Remark}[section]

\title{}
\author{}

\begin{document}
\title{Complete Classification and Nondegeneracy of $N$-Component Cubic Nonlinear Schr\"{o}dinger System  in $\R$}
\author{Yujin Guo\thanks{School of Mathematics and Statistics,  Key Laboratory of Nonlinear Analysis $\&$ Applications (Ministry of Education), Central China Normal University, Wuhan 430079, P. R. China. Y. J. Guo is partially supported by  NSF of China
(Grants 12225106 and 12371113) and National Key R $\&$ D Program of China (Grant 2023YFA1010001). Email: yguo@ccnu.edu.cn. }, \  Yong Luo\thanks{School of Mathematics and Statistics, and Hubei Key Laboratory of
Mathematical Sciences,  Central China Normal University, Wuhan 430079, P. R. China. Email: yluo@ccnu.edu.cn.},\ \,  and\; Juncheng Wei\thanks{Department of Mathematics, Chinese University of Hong Kong, Shatin, NT, Hong Kong. J. Wei is partially supported  by   Hong Kong General Research Funds ``New frontiers in singular limits of nonlinear partial differential equations" and "On critical and supercritical Fujita equations". Email: wei@math.cuhk.edu.hk.}  }

\date{}

\smallbreak \maketitle

\begin{abstract}
We study the one-dimensional cubic nonlinear Schr\"{o}dinger system
\[
u_i''+2\left(\sum_{k=1}^N u_k^2\right)u_i=-\mu_i u_i
\quad \mbox{in } \ \mathbb R,\ \ i=1,2,\cdots,N,
\]
where  $u=(u_1,\cdots,u_N)\in \big(H^1(\mathbb{R})\big)^N$, $\mu_1\leq\mu_2\leq\cdots\leq\mu_N<0$, and $N\geq 2$ is arbitrary. In this paper, we prove the following results for any $N\ge 2$: (i). All nontrivial solutions of the system can be completely classified; (ii). The linearized operator at any nontrivial solution of the system is non-degenerate; (iii).  For all $i=1, 2,\cdots, N$, the exact $L^2$-mass identity of $u_i$ is derived in terms of $2\sqrt {|\mu_i|}$, which yields a complete characterization of normalized solutions satisfying $\int_{\mathbb{R}}u_i^2dx=1$. These settle some conjectures of [R. Frank, D. Gontier and M. Lewin, CMP, 2021] and [Y. Guo, Y. Luo and J. Wei, APDE, 2026], where the system was addressed specially for $N=2$ and $N=3$, respectively.
\end{abstract}

\section{Introduction}
We consider the following $N$-component system
\begin{equation}\label{FM}
-u''_i-2(\sum_{j=1}^N u^2_j)u_i=\mu_iu_i\ \,\ \hbox{in }\ \R,\ \ i=1,2,\cdots,N,
\end{equation}
where \(N\ge2\), \(\mu_1\le \mu_2\le\cdots\le \mu_N<0\), and
\(u_i\in H^1(\mathbb R)\). The systems of coupled nonlinear Schr\"odinger equations arise in several
physical and mathematical contexts, including multi-component
Bose--Einstein condensates, nonlinear optics, and effective models of
many-particle quantum systems, see e.g.
\cite{LS,M,RL,ZS}. In such models, different component
interacts with each other through a common nonlinear density, which often
create the complicated structures absent in a single scalar equation. In the
elliptic PDE theory, much attention has been devoted to the existence, uniqueness
multiplicity and other qualitative properties  of standing waves and
normalized solutions for these systems, see
\cite{A,BS,BZZ,CS,LW,WZZ,WW} and the references therein.  These works also provide
generally  variational backgrounds of the coupled systems.  Instead of studying  a
particular state, in this paper we aim at a complete classification on all finite-energy
solutions of the one-dimensional system \eqref{FM}.

A more specific motivation of studying \eqref{FM} comes from the nonlinear
Schr\"odinger equations of orthonormal functions, which are naturally related
to the effective descriptions of fermionic systems and Lieb--Thirring
inequalities (cf. \cite{Lewin3,Lewin,GLN}).  See also \cite{HLT,LT} for the classical
Lieb--Thirring (L-T) inequality.  The finite-rank version of L-T inequality
is concerned with the estimates of the form
\begin{equation}\label{1:FM}
\sum_{j=1}^N |\lambda_j(-\Delta-V)|^\gamma
\leq L_{\gamma,d}^{(N)}
\int_{\mathbb R^d} V_+^{\gamma+d/2}\,dx ,
\end{equation}
where \(\lambda_j(-\Delta-V)\) denotes the min-max levels of the
Schr\"odinger operator \(-\Delta-V\) in $\R^d$.  The optimizer problem of this
finite-rank inequality was studied in details by Frank-Gontier-Lewin
\cite{Lewin3}. In particular, Frank-Gontier-Lewin showed in \cite{Lewin3} that the optimizers of (\ref{1:FM}) satisfy
an Euler-Lagrange equation, which is a system of coupled nonlinear
Schr\"odinger equations.

For the particular case where \(d=1\) and  the
exponent \(\gamma=3/2\), the optimizer of (\ref{1:FM}) satisfies  the cubic system \eqref{FM}, after a simple normalization
of the eigenfunctions, and the corresponding finite-rank optimizing potential is precisely the
KdV \(N\)-soliton potential, which is reflectionless (cf. \cite{Lewin3}).
Thus, a classification on the solutions for \eqref{FM} gives an
eigenfunction-level description of the bound-state profiles associated with
the soliton potential. This solitonic point of view is already visible in the low-dimensional
classifications. More precisely, when \(N=2\), the classification  on the solutions of \eqref{FM} can be
expressed (cf. \cite{Lewin}) through explicit formulas. When \(N=3\), Guo-Luo-Wei
\cite{GLW2026} recently obtained the complete classification and nondegeneracy of solutions for \eqref{FM}.
We remark that the above mentioned classifications in \cite{Lewin,GLW2026} make full use of  Hirota-type bilinear
methods, together with algebraic  constants of
motion.  These arguments also lead to some conjectures in \cite{GLW2026} on the structure of
solutions for \eqref{FM} in the general case \(N>3\).
However, it should be emphasized that it is nontrivial to extend the cases  \(N=2,3\) to the  general case \(N\ge 2\). Actually, one can note that the proofs for the cases \(N=2,3\) heavily rely on the identities whose explicit
forms become  increasingly complicated as the number $N$ increases.
In particular, a direct use of the conserved quantities and a componentwise
Hirota expansion result in the more involved analysis of algebraic structures for \eqref{FM}.
Thus, even though the formulas suggested by the cases \(N=2,3\) exhibit a clear determinant pattern, a uniform verification and a complete classification for the general case \(N\geq 2\) require a more systematic approach.

The main purpose of the present paper is to provide such a uniform approach to the complete classification and nondegeneracy of solutions for \eqref{FM}.  We shall prove that
all \(H^1(\mathbb R)^N\) solutions of \eqref{FM} can be completely
classified for arbitrary \(N\ge2\).  In particular, when the spectral parameters  $\mu_1<\mu_2<\cdots<\mu_N<0$ are
pairwise distinct, each solution with nontrivial components is shown to be
given by an explicit determinant formula.  When some spectral parameters of $\mu_1\le \mu_2\le \cdots\le\mu_N<0$
coincide with each other, it however does not produce genuinely new scalar profiles, in the sense that the system just reduces to the distinct-parameter case by grouping equal spectral values, and the
only additional freedom is the natural rotation inside each spectral block.
On the other hand, we shall further prove the nondegeneracy of the classified solutions for \eqref{FM}, and their exact
\(L^2\)-mass identities are also derived in terms of $\mu_1\le \mu_2\le \cdots\le\mu_N<0$.  In particular, our results settle some conjectures of \cite{Lewin,GLW2026} on the complete classification and nondegeneracy  of solutions for \eqref{FM} in the general case \(N\ge 2\).

\subsection{Main results}

In this subsection, we shall introduce the main results of the present paper. We begin with the case of pairwise distinct spectral parameters, where the
full determinant structure is already visible.  Since identically zero
components can  be essentially removed, it is enough to state the following theorem for
solutions whose components are all nontrivial.

\begin{theorem}\label{thm1.1}
For $\mu_1<\mu_2<\cdots<\mu_N<0$, suppose $u=(u_1,u_2,\cdots,u_N)$ is a  solution of \eqref{FM} in $(H^1(\R)\setminus\{0\})^N$. Then there exists a vector $a=(a_1,a_2,\cdots,a_N)\in(\R\setminus\{0\})^N$ such that $u$ can be expressed exactly as
\begin{equation}\label{exp1}
	u_i(x)=\frac{a_ie^{\eta_ix}\det(I+B^{(i)}M(x))}{\det(I+M(x))},\ \ i=1,2,\cdots,N,
\end{equation}
where
\begin{equation}\label{MatrixM}
M(x):=\bigl(m_{jk}(x)\bigr)_{1\le j,k\le N},
\ \
m_{jk}(x):=\frac{a_j a_k}{2\sqrt{\eta_j\eta_k}(\eta_j+\eta_k)}e^{(\eta_j+\eta_k)x},
\end{equation}
and for each $i=1,2,\cdots,N$,  $\eta_i=\sqrt{-\mu_i}>0$, and
\begin{equation}\label{Bi}
B^{(i)}:=\operatorname{diag}\Big(
\frac{\eta_i-\eta_1}{\eta_i+\eta_1},
\frac{\eta_i-\eta_2}{\eta_i+\eta_2},
\dots,
\frac{\eta_i-\eta_N}{\eta_i+\eta_N}
\Big).
\end{equation}
\end{theorem}

%\begin{remark}\label{rem:parameters-thm11}
We remark that the parameters \(a_1,\cdots,a_N\) of Theorem~\ref{thm1.1} are not
auxiliary constants introduced by the determinant ansatz. They are precisely
the leading coefficients of the solution $u$ at \(-\infty\). Indeed, since
\(M(x)\to0\) as \(x\to-\infty\), the formula \eqref{exp1} gives that
\[
u_i(x)=a_i e^{\eta_i x}+o(e^{\eta_i x})
\ \ \mbox{as}\ \  x\to-\infty,\ \ i=1,\cdots,N .
\]
Thus, Theorem~\ref{thm1.1} gives a one-to-one parametrization of all
solutions with prescribed nonzero left asymptotic data. Moreover, translations and independent sign changes of Theorem~\ref{thm1.1} are encoded directly in
the parameters. If \(U^a=(U_1^a, \cdots, U_N^a)\) denotes the solution corresponding to some vector \(a\), then
\[
U^a(x+t)=U^{a(t)}(x),
\ \
a_i(t)=a_i e^{\eta_i t},
\]
and for \(\sigma_i\in\{\pm1\}\),
\[
U^{(\sigma_1a_1,\cdots,\sigma_Na_N)}_i
=
\sigma_i U^a_i .
\]
Consequently, the determinant formula of Theorem~\ref{thm1.1} describes the full solution manifold of \eqref{FM},
including its natural translation and sign symmetries.
%\end{remark}

%\begin{remark}\label{rem:spectral-interpretation}
Denote
\[
V_{pot}^a(x):=2\sum_{j=1}^N (U_j^a(x))^2,
\ \
H_a:=-\frac{d^2}{dx^2}-V_{pot}^a(x)\ \ \,\mbox{in} \, \ L^2(\R).
\]
Then each component \(U_i^a\) of the solution \(U^a=(U_1^a, \cdots, U_N^a)\) for \eqref{FM} is an eigenfunction of \(H_a\) with the
eigenvalue \(\mu_i\):
\[
H_aU_i^a=\mu_iU_i^a\ \,\ \mbox{in} \ \ \R,
\ \  i=1,\cdots,N .
\]
In particular, when
\[
\mu_1<\mu_2<\cdots<\mu_N<0,
\]
the determinant formula shows that \(U_1^a\) has a fixed sign on
\(\mathbb R\). Hence,   \(U_1^a\) is a ground state of \(H_a\), and \(\mu_1\) is the
first eigenvalue of \(H_a\).
Since \(H_a\) is self-adjoint, we note that the eigenfunctions corresponding to
distinct eigenvalues are orthogonal in \(L^2(\mathbb R)\), which implies that
\[
\int_{\mathbb R}U_i^aU_j^a\,dx=0,
\ \  i\neq j .
\]
The exact $L^2-$mass identity of Theorem~\ref{thm L2iden} below further gives that
\[
\|U_i^a\|_{L^2(\mathbb R)}^2=2\eta_i=\sqrt{-\mu_i}>0,
\ \  i=1,\cdots,N .
\]
Thus, we have
\[
\sum_{i=1}^N |\mu_i|^{1/2}
=
\sum_{i=1}^N \eta_i
=
\frac14\int_{\mathbb R}V_{pot}^a(x)\,dx .
\]
This spectral viewpoint explains the relevance of the classification for  \eqref{FM} to
the finite-rank Lieb-Thirring inequality studied in
Frank-Gontier-Lewin~\cite{Lewin}.
%\end{remark}

We now sketch  the proof of Theorem~\ref{thm1.1}.
For any \(N\), the main point of proving  Theorem~\ref{thm1.1} is to verify the determinant ansatz
without expanding it one component by one component.  By a determinant-ratio
identity, the formula of Theorem~\ref{thm1.1} can be rewritten as
\[
U_i^a=\sqrt{2\eta_i}\,q_i,
\ \  q=(q_1,\cdots,q_N)=(I+M)^{-1}v,
\]
where \(v\) is a suitable vector.  The verification is then reduced to the claim
that \(q\) satisfies a closed second-order differential system.  This can
be checked uniformly in \(N\) by exploiting the special Cauchy-type
structure of \(M(x)\) in (\ref{MatrixM}).

To finish the proof of Theorem~\ref{thm1.1}, it then remains to prove that no other \(H^1\)-solutions exist. For an arbitrary
solution \(u\), each component equation is an \(L^1\)-perturbation of
\(y''-\eta_i^2y=0\) near \(-\infty\). The \(H^1\)-condition excludes the
growing mode $e^{-\eta_ix}$ as $x\to-\infty$, and hence it determines a unique vector
\(a=(a_1,\cdots,a_N)\) from the asymptotic behavior
\[
u_i(x)=a_i e^{\eta_i x}+o(e^{\eta_i x})
\ \ \mbox{as}\ \  x\to-\infty .
\]
Comparing \(u\) with the determinant solution \(U^a\) associated with the
same asymptotic data, an asymptotic uniqueness principle then gives
\(u\equiv U^a\) on a left half-line, and the usual uniqueness theorem of
ODE systems then yields \(u\equiv U^a\) on \(\mathbb R\). We refer to Sections 2 and 3 for the detailed proof of   Theorem~\ref{thm1.1}.

 Based on the classification of Theorem~\ref{thm1.1}, we next describe the linearized
structure of solutions for \eqref{FM}.
For any fixed vector
$a=(a_1,\cdots,a_N)\in(\mathbb R\setminus\{0\})^N$, let
\begin{equation}\label{Ua}
	U^a=(U_1^a,\cdots,U_N^a)
\end{equation}
be the determinant solution given in Theorem~\ref{thm1.1}. Direct computations then yield that
\[
U_i^a(x)=a_i e^{\eta_i x}+o(e^{\eta_i x})
\ \  \mbox{as} \ \  x\to-\infty .
\]
The linearized operator at $U^a$  is defined as
\begin{equation}\label{LUa}
	(\mathcal L_{U^a}\phi)_i
	:=
	-\phi_i''
	-\mu_i\phi_i
	-2\left(\sum_{j=1}^N (U_j^a)^2\right)\phi_i
	-4U_i^a\sum_{j=1}^N U_j^a\phi_j,
	\ \  i=1,\cdots,N.
\end{equation}
The natural kernel directions of (\ref{LUa}) are obtained by differentiating
the family with respect to the parameters \(a_1,\cdots,a_N\), and the
following nondegeneracy theorem shows that these are the only \(H^1\)-kernel directions.

\begin{theorem}
\label{prop:nondegeneracy}
For \(\mu_1<\mu_2<\cdots<\mu_N<0\),  let \(U^a\) be the
solution given by Theorem~\ref{thm1.1} for some vector
\(a=(a_1,\cdots,a_N)\in(\mathbb R\setminus\{0\})^N\).
Then we have $\dim\ker_{H^1(\mathbb R)^N}\mathcal L_{U^a}=N$, and
\[
\ker_{H^1(\mathbb R)^N}\mathcal L_{U^a}
=
\operatorname{span}
\left\{
\partial_{a_1}U^a,\cdots,\partial_{a_N}U^a
\right\},
\]
where \(\mathcal L_{U^a}\) is defined by \eqref{LUa}.
\end{theorem}

We now pass from the case of pairwise distinct spectral parameters to the general case
of repeated spectral values.
Roughly speaking, we can illustrate that these repeated parameters do not however create new scalar profiles. More precisely, we shall prove that after grouping
equal spectral values, all components in the same spectral block are
multiples of one common profile, with only the natural rotational freedom
inside that block.
Let
\[
\mu_1<\mu_2<\cdots<\mu_K<0
\]
be the distinct spectral values, which have their multiplicities
\(k_1,\cdots,k_K\), so that \(k_1+\cdots+k_K=N\).
%We write
%\[
%\mu_\alpha=-\eta_\alpha^2,\ \  \eta_\alpha>0,
%\]
%and set
Set
\begin{equation}\label{eq:multi-index}
N_\alpha:=k_1+\cdots+k_\alpha,\ \  N_0:=0,
\ \
I_\alpha:=\{N_{\alpha-1}+1,\cdots,N_\alpha\}.
\end{equation}
Let $\nu_1,\cdots,\nu_N$ be the full list of the spectral parameters counted with their
multiplicities, namely
\begin{equation}\label{eq:multi-nu}
\nu_i=\mu_\alpha,
\ \
i\in I_\alpha,\quad \alpha=1,\cdots,K .
\end{equation}
For the reduced distinct-root system
\begin{equation}\label{eq:reduced-system}
V_\alpha''
+2\left(\sum_{\beta=1}^K V_\beta^2\right)V_\alpha
+\mu_\alpha V_\alpha=0,
\ \  \alpha=1,\cdots,K,
\end{equation}
let
\[
V^a=(V_1^a,\cdots,V_K^a),
\ \
a=(a_1,\cdots,a_K)\in(\mathbb R\setminus\{0\})^K,
\]
be the determinant family given by Theorem~\ref{thm1.1}.

For
\[
n^{(\alpha)}
=
(n_i^{(\alpha)})_{i\in I_\alpha}
\in \mathbb S^{k_\alpha-1},
\ \  \alpha=1,\cdots,K,
\]
namely,
\[
\sum_{i\in I_\alpha} (n_i^{(\alpha)})^2=1,
\]
define
\begin{equation}\label{eq:multi-lift}
U_i^{a,n}(x)=n_i^{(\alpha)}V_\alpha^a(x),
\ \
i\in I_\alpha,\ \ \alpha=1,\cdots,K.
\end{equation}
We denote the extended family by
\begin{equation}\label{eq:multi-family}
\mathcal M_{\mathrm{multi}}
:=
\left\{
U^{a,n} \text{ defined by } \eqref{eq:multi-lift}:\,
a\in(\mathbb R\setminus\{0\})^K,\
n^{(\alpha)}\in\mathbb S^{k_\alpha-1}
\right\}.
\end{equation}
We shall derive the following property.

\begin{theorem}
\label{thm:completeness-multiple-root}
Let $I_\alpha$ and $\nu_1,\cdots,\nu_N$ be defined by \eqref{eq:multi-index} and \eqref{eq:multi-nu},  respectively, and assume $
u=(u_1,\cdots,u_N)\in H^1(\mathbb R)^N$
is a solution of
\begin{equation}\label{eq:multi-system}
	u_i''
	+2\left(\sum_{j=1}^N u_j^2\right)u_i
	+\nu_i u_i=0\ \ in \,\ \R,
	\ \  i=1,\cdots,N.
\end{equation}
If \(\sum_{i\in I_\alpha}u_i^2\not\equiv0\) holds for every
\(\alpha=1,\cdots,K\), then \(u\in\mathcal M_{\mathrm{multi}}\), where
\(\mathcal M_{\mathrm{multi}}\) is given by \eqref{eq:multi-family}.
\end{theorem}

\begin{remark}\label{rem:conj51}
One can note that Theorems~\ref{thm1.1} and
\ref{thm:completeness-multiple-root} give a complete classification
on \(H^1(\mathbb R)^N\) solutions of \eqref{FM} for any
\(N\ge2\). In particular, in the case of pairwise distinct spectral
parameters, Theorem~\ref{thm1.1} proves the explicit formula conjectured in
\cite[Conjecture~5.1]{GLW2026}.
\end{remark}

We next describe the nondegeneracy of the extended family
$\mathcal M_{\mathrm{multi}}$ for the general case  $\nu_1\le\nu_2\le\cdots \le\nu_N<0$. Let
$U^{a,n}\in\mathcal M_{\mathrm{multi}}$
be given by \eqref{eq:multi-lift}, namely
\begin{equation}\label{eq:multi-U}
U^{a,n}_i(x)=n_i^{(\alpha)}V_\alpha^a(x),
\ \
i\in I_\alpha,\quad \alpha=1,\cdots,K.
\end{equation}
Since $n^{(\alpha)}\in \mathbb S^{k_\alpha-1}$, we have
\begin{equation}\label{eq:mul-S}
S^{a}=\sum_{\ell=1}^N (U^{a,n}_\ell)^2
=
\sum_{\beta=1}^K (V_\beta^a)^2 .
\end{equation}
The linearized operator at $U^{a,n}$ is  defined by
\begin{equation}\label{eq:multi-LU}
(\mathcal L_{U^{a,n}}\phi)_i
:=
-\phi_i''
-\nu_i\phi_i
-2S^{a}\phi_i
-4U^{a,n}_i\sum_{\ell=1}^N U^{a,n}_\ell\phi_\ell\ \ \,\mbox{in} \,\ \R,
\ \  i=1,\cdots,N.
\end{equation}
In view of \eqref{eq:multi-nu}, this can equivalently be written as
\begin{equation}\label{eq:multi-LU-block}
(\mathcal L_{U^{a,n}}\phi)_i
=
-\phi_i''
-\mu_\alpha\phi_i
-2S^{a}\phi_i
-4U^{a,n}_i\sum_{\ell=1}^N U^{a,n}_\ell\phi_\ell\ \ \,\mbox{in} \,\ \R,
\ \  i\in I_\alpha .
\end{equation}

The linearized operator at the reduced solution $V^a$ is  defined by
\begin{equation}\label{eq:multi-LV}
(\mathcal L_{V^a}p)_\alpha
:=
-p_\alpha''
-\mu_\alpha p_\alpha
-2S^{a}p_\alpha
-4V_\alpha^a\sum_{\beta=1}^K V_\beta^a p_\beta\ \ \,\mbox{in} \,\ \R,
\ \  \alpha=1,\cdots,K.
\end{equation}
Applying Theorem~\ref{prop:nondegeneracy}
to the reduced distinct-root system \eqref{eq:reduced-system}, we have
\begin{equation}\label{eq:reduced-nondeg}
\ker_{H^1(\mathbb R)^K}\mathcal L_{V^a}
=
\operatorname{span}
\left\{
\partial_{a_1}V^a,\cdots,\partial_{a_K}V^a
\right\}.
\end{equation}
The nondegeneracy  in the repeated-parameter case has two types
of kernel directions.  The first type comes from changing the reduced
parameters \(a_\gamma\).  The second type comes from rotating the unit
vectors \(n^{(\alpha)}\) inside the spectral blocks.
The tangent space $T_{U^{a,n}}\mathcal M_{\mathrm{multi}}$ is defined as the span of the
following directions. Firstly, set for $\gamma=1,\cdots,K$,
\begin{equation}\label{eq:tangent-a-direction}
\Phi_i^{(\gamma)}
=
n_i^{(\alpha)}\partial_{a_\gamma}V_\alpha^a,
\ \
i\in I_\alpha,\quad \alpha=1,\cdots,K.
\end{equation}
Secondly, for each $\alpha=1,\cdots,K$ and each tangent vector
$\xi$ to the unit sphere $\mathbb S^{k_\alpha-1}$ at $n^{(\alpha)}$, namely
\[
\xi=(\xi_i)_{i\in I_\alpha}\in T_{n^{(\alpha)}}\mathbb S^{k_\alpha-1}
=
\left\{\zeta\in\mathbb R^{k_\alpha}:\quad\zeta\cdot n^{(\alpha)}=0\right\},
\]
set
%Second, for each $\alpha=1,\cdots,K$ and each
%\[
%\xi=(\xi_i)_{i\in I_\alpha}\in \mathbb R^{k_\alpha},
%\ \
%\xi\cdot n^{(\alpha)}=0,
%\]
\begin{equation}\label{eq:tangent-rotation-direction}
\Psi_i^{(\alpha,\xi)}
=
\begin{cases}
	\xi_i V_\alpha^a, & i\in I_\alpha,\\
	0, & i\notin I_\alpha.
\end{cases}
\end{equation}
Thus
\begin{equation}\label{eq:tangent-space}
T_{U^{a,n}}\mathcal M_{\mathrm{multi}}
=
\operatorname{span}
\left\{
\Phi^{(\gamma)},\ \Psi^{(\alpha,\xi)}:
\gamma=1,\cdots,K,\
\alpha=1,\cdots,K,\
\xi\cdot n^{(\alpha)}=0
\right\}.
\end{equation}
Together with Theorem \ref{prop:nondegeneracy}, the following theorem addresses the  nondegeneracy  for the general case  $\nu_1\le\cdots\le\nu_N<0$.

\begin{theorem}
\label{prop:nondegeneracy-multiple-roots}
Let $U^{a,n}\in\mathcal M_{\mathrm{multi}}$ and $\mathcal L_{U^{a,n}}$ be given by
\eqref{eq:multi-U} and
\eqref{eq:multi-LU}, respectively.
%Assume that, for every $\alpha=1,\cdots,K$,
%\[
%\sum_{i\in I_\alpha}(U_i^{a,n})^2(x)\not\equiv 0,
%\]
%where $I_\alpha$ is defined by \eqref{eq:multi-index}.
Then we have
$\dim\ker_{H^1(\mathbb R)^N}\mathcal L_{U^{a,n}}=N$ and
\[
\ker_{H^1(\mathbb R)^N}\mathcal L_{U^{a,n}}
=
T_{U^{a,n}}\mathcal M_{\mathrm{multi}},
\]
where $T_{U^{a,n}}\mathcal M_{\mathrm{multi}}$ is given by \eqref{eq:tangent-space}.
\end{theorem}

%The next identity is obtained by reorganizing the constants of motion
%introduced in \cite{GLW2026}. More precisely, these constants
%imply a pointwise polynomial relation involving the solution itself and its
%first derivative. This leads to the following new algebraic identity, which
%will play an important role in the proof of the \(L^2\)-mass formulas below.
%\begin{theorem}\label{thm:algiden}
%Under the same assumptions as in
%Theorem~\ref{thm:completeness-multiple-root}, for every \(\alpha=1,\cdots,K\), we have
%\begin{equation}\label{eq:algebraic-identity}
%	\sum_{j\in I_\alpha}
%	\left[
%	(u_j')^2+\big(\sum_{l=1}^N u_l^2+\mu_\alpha\big)u_j^2
%	\right]
%	+
%	\sum_{\beta\neq\alpha}
%	\sum_{\substack{j\in I_\alpha\\ l\in I_\beta}}
%	\frac{(u_j'u_l-u_ju_l')^2}{\mu_\beta-\mu_\alpha}
%	=0 .
%\end{equation}
%\end{theorem}
%
We finally derive the following exact $L^2$-mass identity of   solutions for \eqref{FM}, whose proof is based on an algebraic identity generated by the constants of motion, rather than on the explicit determinant formula of solutions for \eqref{FM}.

\begin{theorem}\label{thm L2iden}
Under the same assumptions of
Theorem~\ref{thm:completeness-multiple-root}, any solution $u=(u_1, \cdots, u_N)$ of \eqref{FM} satisfies for every
\(\alpha=1,\cdots,K\),
\begin{equation}\label{L2ident}
	\int_{\mathbb R}\sum_{i\in I_\alpha}u_i^2(x)\,dx
	=
	2\sqrt{|\mu_\alpha|}>0,
\end{equation}
where the set $I_\alpha$ is defined by \eqref{eq:multi-index}.
\end{theorem}

%We finally record a consequence of the classified structure: each spectral
%block has an exact \(L^2\)-mass determined only by the corresponding
%spectral parameter.  Although the classification above is expressed through
%the determinant representation, the proof of this identity does not rely on
%that formula.  It follows directly from the system itself.  More precisely,
%a pointwise algebraic identity, recorded in
%Remark~\ref{rem:algebraic-identity} as a by-product of the proof, yields a
%first-order differential relation whose limits at \(\pm\infty\) determine
%the mass.

%\begin{remark}\label{rem:algebraic-identity}
The proof of Theorem~\ref{thm L2iden} yields the
following pointwise algebraic identity: Under the same assumptions of
Theorem~\ref{thm:completeness-multiple-root}, we have for every
\(\alpha=1,\cdots,K\),
\begin{equation*}
	\sum_{j\in I_\alpha}
	\left[
	(u_j')^2+
	\left(\sum_{l=1}^N u_l^2+\mu_\alpha\right)u_j^2
	\right]
	+
	\sum_{\beta\neq\alpha}
	\sum_{\substack{j\in I_\alpha\\ l\in I_\beta}}
	\frac{(u_j'u_l-u_ju_l')^2}{\mu_\beta-\mu_\alpha}
	=0 .
\end{equation*}
This identity is derived from the spectral polynomial identity of
Lemma~\ref{lem:polynomial-identity-function} in Section~5.
%
%It is not
%needed for the classification result, but is useful for deriving the exact
%\(L^2\)-mass formulas above.
%\end{remark}

\begin{remark}\label{rem:normalized-solutions}
Theorem~\ref{thm L2iden} also gives a simple description of the
componentwise normalized solutions for \eqref{FM}. Indeed, if a solution
\(u\in\mathcal M_{\mathrm{multi}}\) of \eqref{FM} is written as
\[
u_i=n_i^{(\alpha)}V_\alpha^a,
\ \  i\in I_\alpha,\quad \alpha=1,\cdots,K,
\]
then Theorem~\ref{thm L2iden} gives that
\[
\int_{\mathbb R}u_i^2\,dx
=
\bigl(n_i^{(\alpha)}\bigr)^2
\int_{\mathbb R}(V_\alpha^a)^2\,dx
=
2\eta_\alpha \bigl(n_i^{(\alpha)}\bigr)^2 .
\]
Hence \(u=(u_1,\cdots,u_N)\) satisfies
\[
\int_{\mathbb R}u_i^2\,dx=1,
\ \  i=1,\cdots,N,
\]
if and only if
\[
\eta_\alpha=\sqrt{|\mu_\alpha|}=\frac{k_\alpha}{2}
\quad\text{and}\quad
\bigl(n_i^{(\alpha)}\bigr)^2=\frac1{k_\alpha},
\ \  i\in I_\alpha,\quad \alpha=1,\cdots,K .
\]
\end{remark}

\begin{remark}
As a consequence of Theorems~\ref{thm:completeness-multiple-root} and
\ref{thm L2iden}, the system \eqref{FM} does not admits any orthonormal solution, whose components
are pairwise orthogonal and have the same nonzero \(L^2\)-norm. This addresses the nonexistence of orthonormal solutions for \eqref{FM}, which was conjectured %after Theorem~8 of
by Frank-Gontier-Lewin~\cite[page 1793]{Lewin} and
 Guo-Luo-Wei~\cite[Conjecture~5.2]{GLW2026} for the general case \(N\ge 2\).
 \end{remark}

This paper is organized as follows. In Section~2, we construct a determinant family and verify the equations in the distinct-parameter case. Section~3 is devoted to the proofs of Theorems~\ref{thm1.1} and \ref{prop:nondegeneracy} on the classification and nondegeneracy of solutions for \eqref{FM} in the distinct-parameter case. In Section~4, we prove Theorems~\ref{thm:completeness-multiple-root} and \ref{prop:nondegeneracy-multiple-roots} on the classification and nondegeneracy results in the case of repeated spectral parameters.  Theorem~\ref{thm L2iden} is finally proved in Section~5 by deriving an algebraic identity and the exact $L^2$-mass formulas.

\section{Construction   of Determinant Solutions}
In this section, we verify that the explicit determinant formula stated in
Theorem \ref{thm1.1} satisfies \eqref{FM}. Our argument is divided into two
parts. In Subsection 2.1,  we rewrite the explicit formula \eqref{exp1} for
\(u=(u_1,u_2,\cdots,u_N)\) in terms of an auxiliary vector \(q\). In
Subsection 2.2, we use this vector formulation to verify the differential
equation. The key point is that the auxiliary vector \(q\) satisfies a closed second-order
system, which is equivalent to \eqref{FM} after a simple rescaling.

We now introduce the notation used in the determinant formula. Throughout this section, we assume that
\[
\mu_1<\mu_2<\cdots<\mu_N<0
\quad\hbox{and}\quad
\eta_i=\sqrt{-\mu_i},\quad i=1,2,\cdots,N.
\]
Fixing
$
a=(a_1,\cdots,a_N)\in(\mathbb R\setminus\{0\})^N,
$
we recall that
\[
m_{jk}(x)
=
\frac{a_j a_k}{2\sqrt{\eta_j\eta_k}(\eta_j+\eta_k)}
e^{(\eta_j+\eta_k)x},
\ \
M(x)=(m_{jk}(x))_{1\le j,k\le N},
\]
and
\[
B^{(i)}
=
\operatorname{diag}\left(
\frac{\eta_i-\eta_1}{\eta_i+\eta_1},
\cdots,
\frac{\eta_i-\eta_N}{\eta_i+\eta_N}
\right).
\]
Define
\begin{equation}\label{f-formula}
	f(x):=\det\bigl(I+M(x)\bigr)
\end{equation}
and
\begin{equation}\label{g-formula}
	g_i(x):=a_i e^{\eta_i x}\det\bigl(I+B^{(i)}M(x)\bigr),
	\ \  i=1,\cdots,N,
\end{equation}
such that
\[
u_i(x)=\frac{g_i}{f}\quad\hbox{for}\,\ i=1,2,\cdots,N.
\]
We shall derive the following precise verification.

\begin{proposition}\label{prop:determinant-verification}
For
$\mu_1<\mu_2<\cdots<\mu_N<0$ and $\eta_i=\sqrt{-\mu_i}$, let \(a=(a_1,\cdots,a_N)\in(\R\setminus\{0\})^N\), and define \(f\) and \(g_i\)
by \eqref{f-formula} and \eqref{g-formula}, respectively. Then
\[
u_i(x):=\frac{g_i(x)}{f(x)},\ \  i=1,\cdots,N,
\]
defines a classical solution $u=(u_1,\cdots,u_N)\in \big(H^1(\R)\setminus\{0\}\big)^N$ of \eqref{FM}.
\end{proposition}

The proof of Proposition \ref{prop:determinant-verification} is divided into two steps.  We first establish in Subsection 2.1 a determinant
ratio identity which rewrites \(g_i/f\) in terms of \(q\).  We then verify in Subsection 2.2
that this vector \(q\) satisfies the closed differential system equivalent
to \eqref{FM}.

\subsection{A determinant ratio formula}

We begin with the introduction of an auxiliary vector.
Set
\begin{equation}\label{vvect}
	v(x):=
	\left(
	v_1(x),\cdots,v_N(x)
	\right)^T,
	\ \
	v_j(x):=\frac{a_j}{\sqrt{2\eta_j}}e^{\eta_jx}.
\end{equation}
Then $M(x)=V(x)C V(x)$, where
\begin{equation}\label{Vm}
	V(x):=\operatorname{diag}(v_1(x),\cdots,v_N(x)),
	\ \
	C:=\left(\frac{1}{\eta_j+\eta_k}\right)_{1\le j,k\le N}.
\end{equation}
We first observe that \(I+M(x)\) is invertible for every \(x\in\mathbb R\).
Indeed,
note that
\[
\frac1{\eta_j+\eta_k}
=
\int_0^\infty e^{-\eta_jt}e^{-\eta_kt}\,dt,
\]
and thus \(C\) is the Gram matrix of the functions
\(e^{-\eta_jt}\) in \(L^2(0,\infty)\). Since the numbers
\(\eta_j\) are distinct, these functions are linearly independent, and
therefore \(C\) is positive definite. Since \(V(x)\) is invertible, \(M(x)=V(x)CV(x)\) is positive definite by
congruence. In particular, \(I+M(x)\) is positive definite and hence
invertible for every \(x\in\mathbb R\).

Since $I+M(x)$ is invertible,  define
\begin{equation}\label{eq:q-definition}
q(x):=(I+M(x))^{-1}v(x),
\end{equation}
We now prove the following lemma.

\begin{lemma}\label{lem:ratio-formula}
Let \(f(x)\), \(g_i(x)\)  and \(q(x)\) be defined by
\eqref{f-formula}, \eqref{g-formula}, and \eqref{eq:q-definition}.  Set $u_i(x):=\frac{g_i(x)}{f(x)}$, then for every
\(i=1,\cdots,N\),
\[
u_i(x)=
\sqrt{2\eta_i}\,q_i(x),
\]
where \(q_i\) denotes the \(i\)-th component of \(q\).
\end{lemma}

\begin{proof}
For any \(i\in\{1,\cdots,N\}\), we prove the identity by Cramer's rule.
For simplicity, we suppress the variable \(x\) in the following computation.
Set
\[
A:=I+M.
\]
Let \(A_i\) be the matrix obtained from \(A\) by replacing its \(i\)-th
column by \(v\). Since \(q=A^{-1}v\), Cramer's rule gives that
\begin{equation}\label{eq:q-cramer}
	q_i=\frac{\det A_i}{\det A}.
\end{equation}
Thus, it remains to identify the numerator \(\det A_i\). We claim that
\begin{equation}\label{eq:Ai-main-identity}
	\det A_i
	=
	v_i\det\bigl(I+B^{(i)}M\bigr).
\end{equation}
Once this is proved, the desired conclusion follows immediately from the facts that
\(f=\det A\) and
\[
a_i e^{\eta_i x}=\sqrt{2\eta_i}\,v_i(x).
\]

We now prove \eqref{eq:Ai-main-identity}. Write
\[
\beta_j^{(i)}
:=
\frac{\eta_i-\eta_j}{\eta_i+\eta_j},
\ \
B^{(i)}=\operatorname{diag}\bigl(\beta_1^{(i)},\cdots,\beta_N^{(i)}\bigr).
\]
For any \(N\times N\) matrix \(K\) and any subset
\(I\subset\{1,\cdots,N\}\), let \(K[I]\) denote the principal submatrix of
\(K\) indexed by \(I\), where the elements of \(I\) are always taken in their
natural increasing order. We also use the convention
\(\det K[\emptyset]=1\). Then we have the standard principal minor expansion (see, e.g., \cite[(1.2.13)]{HornJohnson2013}):
\[
\det(I+K)=\sum_{I\subset\{1,\cdots,N\}}\det K[I].
\]
Applying this identity to \(K=B^{(i)}M\), and using the fact that
\(B^{(i)}\) is diagonal, we obtain that
\begin{equation}\label{eq:BiM-principal-expansion}
	\det\bigl(I+B^{(i)}M\bigr)
	=
	\sum_{I\subset\{1,\cdots,N\}\setminus\{i\}}
	\left(\prod_{j\in I}\beta_j^{(i)}\right)\det M[I].
\end{equation}
Indeed, all terms with \(i\in I\) vanish because \(\beta_i^{(i)}=0\).
Since
\[
M_{kj}=\frac{v_kv_j}{\eta_k+\eta_j},
\]
we have
\[
M[I]
=
\operatorname{diag}(v_j)_{j\in I}
\left(\frac{1}{\eta_k+\eta_j}\right)_{k,j\in I}
\operatorname{diag}(v_j)_{j\in I}.
\]
Hence
\begin{equation}\label{eq:MI-cauchy-form}
	\det M[I]
	=
	\left(\prod_{j\in I}v_j^2\right)
	\det\left(\frac{1}{\eta_k+\eta_j}\right)_{k,j\in I}.
\end{equation}
Combining \eqref{eq:BiM-principal-expansion} and
\eqref{eq:MI-cauchy-form}, we obtain that
\begin{equation}\label{eq:BiM-expansion}
	\det\bigl(I+B^{(i)}M\bigr)
	=
	\sum_{I\subset\{1,\cdots,N\}\setminus\{i\}}
	\left(\prod_{j\in I}\beta_j^{(i)}\right)
	\left(\prod_{j\in I}v_j^2\right)
	\det\left(\frac{1}{\eta_k+\eta_j}\right)_{k,j\in I}.
\end{equation}

We next expand \(\det A_i\) by the multilinearity of the determinant with
respect to the columns. For each \(j\neq i\), the \(j\)-th column of \(A_i\)
is
\[
e_j+M_{\cdot j},
\]
where \(e_j\) is the \(j\)-th column of the identity matrix and
\(M_{\cdot j}\) is the \(j\)-th column of \(M\). The \(i\)-th column of
\(A_i\) is \(v\). Thus, in the multilinear expansion,
we choose either \(e_j\) or \(M_{\cdot j}\) for each \(j\neq i\). Denote by
\(I\subset\{1,\cdots,N\}\setminus\{i\}\) the set of indices for which
\(M_{\cdot j}\) is chosen. Then
\begin{equation}\label{eq:Ai-expansion}
	\det A_i
	=
	\sum_{I\subset\{1,\cdots,N\}\setminus\{i\}} \Delta_I,
\end{equation}
where \(\Delta_I\) is the determinant of the matrix whose \(j\)-th column is
\(M_{\cdot j}\) for \(j\in I\), is \(e_j\) for
\(j\notin I\cup\{i\}\), and is \(v\) for \(j=i\).

To compute \(\Delta_I\), set
\[
S:=I\cup\{i\}.
\]
For \(j\notin S\), the \(j\)-th column is the unit vector \(e_j\). Hence this column has only one nonzero entry, namely the entry \(1\) in the \(j\)-th row. Expanding the determinant successively along these unit columns, we may delete all rows and columns with labels \(j\notin S\). Therefore \(\Delta_I\) is equal to the determinant of the matrix obtained by keeping only the rows and columns whose labels belong to \(S\). In this reduced matrix, the column with label \(i\) is \((v_k)_{k\in S}\), while the column with label \(j\in I\) is \((M_{kj})_{k\in S}\).

We then move the row and the column with label \(i\) simultaneously to the
first position. The row permutation and the column permutation have the same
sign, so their signs cancel. Thus the determinant is unchanged. We then obtain that
\begin{equation}\label{eq:DeltaI-block}
	\Delta_I
	=
	\det
	\begin{pmatrix}
		v_i & (M_{ij})_{j\in I}\\
		(v_k)_{k\in I} & (M_{kj})_{k,j\in I}
	\end{pmatrix}.
\end{equation}
Using \(M_{kj}=v_kv_j/(\eta_k+\eta_j)\), we factor \(v_k\) from each row
with label \(k\in S\), and \(v_j\) from each column with label \(j\in I\).
This gives that
\begin{equation}\label{eq:DeltaI-factor}
	\Delta_I
	=
	v_i\left(\prod_{j\in I}v_j^2\right)
	\det
	\begin{pmatrix}
		1 & \left(\frac{1}{\eta_i+\eta_j}\right)_{j\in I}\\
		1 & \left(\frac{1}{\eta_k+\eta_j}\right)_{k,j\in I}
	\end{pmatrix}.
\end{equation}
In the last determinant, subtract the first row from each of the remaining
rows and then expand along the first column. This gives
\[
\det
\begin{pmatrix}
	1 & \left(\frac{1}{\eta_i+\eta_j}\right)_{j\in I}\\
	1 & \left(\frac{1}{\eta_k+\eta_j}\right)_{k,j\in I}
\end{pmatrix}
=
\det
\left(
\frac{1}{\eta_k+\eta_j}
-
\frac{1}{\eta_i+\eta_j}
\right)_{k,j\in I}.
\]
Moreover, the matrix on the right-hand side admits the factorization
\[
\left(
\frac{1}{\eta_k+\eta_j}
-
\frac{1}{\eta_i+\eta_j}
\right)_{k,j\in I}
=
\operatorname{diag}(\eta_i-\eta_k)_{k\in I}
\left(\frac{1}{\eta_k+\eta_j}\right)_{k,j\in I}
\operatorname{diag}\left(\frac{1}{\eta_i+\eta_j}\right)_{j\in I}.
\]
We thus conclude from above that
\begin{equation}\label{eq:DeltaI-final}
\Delta_I
=
v_i
\left(\prod_{j\in I}
\frac{\eta_i-\eta_j}{\eta_i+\eta_j}
\right)
\left(\prod_{j\in I}v_j^2\right)
\det\left(\frac{1}{\eta_k+\eta_j}\right)_{k,j\in I}.
\end{equation}

Substituting \eqref{eq:DeltaI-final} into \eqref{eq:Ai-expansion}, we get that
\begin{equation*}
\det A_i
=
\sum_{I\subset\{1,\cdots,N\}\setminus\{i\}}
v_i\left(\prod_{j\in I}\beta_j^{(i)}\right)
\left(\prod_{j\in I}v_j^2\right)
\det\left(\frac{1}{\eta_k+\eta_j}\right)_{k,j\in I}.
\end{equation*}
Combining this with \eqref{eq:BiM-expansion} gives
\eqref{eq:Ai-main-identity}. Finally, since \(f=\det A\), it follows from \eqref{eq:q-cramer}  and
\eqref{eq:Ai-main-identity} that
\[
q_i(x)
=
\frac{v_i(x)\det(I+B^{(i)}M(x))}{f(x)}.
\]
Since \(a_i e^{\eta_i x}=\sqrt{2\eta_i}\,v_i(x)\), the desired identity follows.
This completes the proof.
\end{proof}

\subsection{Verification of the determinant family }
In this Subsection we shall complete the proof of Proposition~\ref{prop:determinant-verification}.
By Lemma~\ref{lem:ratio-formula}, the determinant ratios are reduced to
the vector
\[
q=(I+M)^{-1}v,
\ \
u_i=\sqrt{2\eta_i}\,q_i .
\]
It remains to show that \(q\) satisfies the closed second-order system
equivalent to \eqref{FM}.
\begin{proof}[\normalfont\bfseries Proof of Proposition~\ref{prop:determinant-verification}]
By Lemma~\ref{lem:ratio-formula}, the determinant solution satisfies
\[
u_i=\sqrt{2\eta_i}\,q_i.
\]
Thus the proof reduces to deriving a closed equation for \(q\).
Let
\[
\Lambda=\operatorname{diag}(\eta_1,\dots,\eta_N).
\]
By \eqref{vvect} and \eqref{MatrixM}, we have
\[
m_{jk}(x)=\frac{v_j(x)v_k(x)}{\eta_j+\eta_k},
\]
and hence
\[
M'(x)=v(x)v(x)^T,
\ \
v'(x)=\Lambda v(x).
\]
Recall that
\[
A(x)=I+M(x)\quad\hbox{and}\quad q(x)=A(x)^{-1}v(x).
\]
Set
\[
s:=v^\top q,\ \  r:=q^\top\Lambda q.
\]
To prove  Proposition~\ref{prop:determinant-verification}, it suffices to show that
\begin{equation}\label{key2.1}
Aq''=A(\Lambda^2q-4rq)
\end{equation}
holds true.
Indeed, since $A=I+M$ is invertible, if \eqref{key2.1} holds true, then we have
\begin{equation}\label{qsecond-final-thm}
q''=\Lambda^2q-4rq.
\end{equation}
Since \(\mu_i=-\eta_i^2\) and
\[
\sum_{j=1}^N u_j^2
=
2\sum_{j=1}^N\eta_jq_j^2=2r,
\]
the equation
\[
q_i''-\eta_i^2q_i+4rq_i=0
\]
is exactly equivalent to \eqref{FM}.
Thus, $u=(u_1,u_2,\cdots,u_N)=(\sqrt{2\eta_1}q_1,\sqrt{2\eta_2}q_2,\cdots,\sqrt{2\eta_N}q_N)$ satisfies \eqref{FM}.

We next prove \eqref{key2.1} as follows.

\medskip
\noindent
\textbf{Step 1. Computation of $Aq''$.}
We claim that
\begin{equation}\label{qsecond2-thm}
Aq''=\Lambda^2v-s\Lambda v-3rv+\frac{s^2}{2}v.
\end{equation}

Indeed, since
\[
Aq=v,
\]
we differentiate both sides of the identity   to obtain that
\[
Aq'+A'q=v'.
\]
Since $A'=vv^\top$ and $v'=\Lambda v$, while
\[
A'q=vv^\top q=v(v^\top q)=(v^\top q)v=sv,
\]
we obtain that
\begin{equation}\label{qprime-eq-thm}
Aq'=(\Lambda-sI)v.
\end{equation}
Differentiating \eqref{qprime-eq-thm} yields that
\[
Aq''+A'q'
=
(\Lambda-sI)'v+(\Lambda-sI)v'
=
-s'v+(\Lambda-sI)\Lambda v.
\]
Since $A'=vv^\top$, this becomes
\begin{equation}\label{qsecond1-thm}
Aq''+v(v^\top q')
=
\Lambda^2v-s\Lambda v-s'v.
\end{equation}

Next, we compute $v^\top q'$ and $s'$. Since
\[v=Aq=(I+M)q=q+Mq,\]
 we have
\[
q^\top\Lambda v
=
q^\top\Lambda q+q^\top\Lambda Mq
=
r+q^\top\Lambda Mq.
\]
Using \(m_{jk}=v_jv_k/(\eta_j+\eta_k)\) and the fact that
the scalar
\(q^\top\Lambda Mq\) satisfies
\[
q^\top\Lambda Mq=(q^\top\Lambda Mq)^\top=q^\top M\Lambda q,
\]
we get
\[
2q^\top\Lambda Mq
=
q^\top(\Lambda M+M\Lambda)q
=
\sum_{j,k=1}^N v_jq_jv_kq_k
=
s^2.
\]
Hence
\begin{equation}\label{def:qLq}
	q^\top\Lambda v
	=
	r+\frac{s^2}{2}.
\end{equation}
Multiplying \eqref{qprime-eq-thm} by \(q^\top\) from the left and
using \(q^\top A=v^\top\), we obtain that
\[
v^\top q'
=
q^\top(\Lambda-sI)v
=
q^\top\Lambda v-s^2.
\]
By \eqref{def:qLq}, this gives that
\begin{equation}\label{vtqprime-thm}
	v^\top q'
	=
	r-\frac{s^2}{2}.
\end{equation}

Finally, differentiating \(s=v^\top q\), we have
\[
s'
=
(v^\top q)'
=
(v')^\top q+v^\top q'
=v^\top\Lambda q +v^\top q'
=q^\top\Lambda v+v^\top q'.
\]
We get from \eqref{def:qLq} and \eqref{vtqprime-thm} that
\begin{equation}\label{sprime2-thm}
	s'=2r.
\end{equation}
Substituting \eqref{sprime2-thm} and \eqref{vtqprime-thm} into
\eqref{qsecond1-thm}, it yields that
\[
Aq''
=
\Lambda^2v-s\Lambda v-s'v-v(v^\top q')
=
\Lambda^2v-s\Lambda v-3rv+\frac{s^2}{2}v,
\]
which proves \eqref{qsecond2-thm}.

\medskip
\noindent
\textbf{Step 2: Computation of $A(\Lambda^2q-4rq)$.}
We claim that
\begin{equation}\label{def:cla1}
A(\Lambda^2q-4rq)
=
\Lambda^2v-s\Lambda v-3rv+\frac{s^2}{2}v.	
\end{equation}

Indeed, since $Aq=v$,
\[
A(\Lambda^2q-4rq)=A\Lambda^2q-4rv,
\]
 and $A=I+M$, we have
\[
\begin{aligned}
A\Lambda^2q=\Lambda^2q+M\Lambda^2q&=\Lambda^2(v-Mq)+M\Lambda^2q\\
&=\Lambda^2v-(\Lambda^2M-M\Lambda^2)q
\end{aligned}
\]
For the commutator term, its $i$-th component is
\[
\bigl((\Lambda^2M-M\Lambda^2)q\bigr)_i
=
\sum_{k=1}^N
(\eta_i^2-\eta_k^2)\frac{v_iv_k}{\eta_i+\eta_k}q_k
=
\sum_{k=1}^N
(\eta_i-\eta_k)v_iv_kq_k.
\]
Hence
\[
\bigl((\Lambda^2M-M\Lambda^2)q\bigr)_i
=
\eta_i s\,v_i-v_i\sum_{k=1}^N \eta_k v_k q_k.
\]
Since
\[
\sum_{k=1}^N \eta_k v_k q_k=q^\top \Lambda v=r+\frac{s^2}{2},
\]
we conclude that
\[
(\Lambda^2M-M\Lambda^2)q
=
s\Lambda v-\left(r+\frac{s^2}{2}\right)v.
\]
Thus,
\[
A\Lambda^2q
=
\Lambda^2v-s\Lambda v+\left(r+\frac{s^2}{2}\right)v,
\]
and hence \eqref{def:cla1} holds true.

We now conclude from \eqref{qsecond2-thm} and \eqref{def:cla1} that \eqref{key2.1} holds true. This completes the proof.
\end{proof}

\section{Solutions for the Distinct-Parameter Case}
In this section, we prove   Theorems~\ref{thm1.1} and \ref{prop:nondegeneracy} on  the classification and nondegeneracy results
in the case where the spectral parameters are distinct.  Our proofs are given  by  comparing solutions through their leading asymptotic
data at \( -\infty \).  We first show that every
\(H^1\)-solution has a well-defined leading coefficient at \(-\infty\).
These coefficients are then   used as the parameters of the determinant
solution constructed in Section~2.  A left-uniqueness principle finally
shows that two solutions with the same leading left asymptotics must
coincide with each other.

\subsection{Left asymptotics and asymptotic uniqueness}
In this subsection, we mainly prove the following scalar ODE lemma to identify the leading asymptotic
behavior of each component \(u_i\) of a solution
\(u=(u_1,\cdots,u_N)\) for \eqref{FM} as \(x\to-\infty\).

\begin{lemma}
\label{lem:L1-asymptotic}
Let \(\eta>0\),
\(q\in C((-\infty,R))\cap L^1(-\infty,R)\), and
\(h\in C((-\infty,R))\). Suppose
\(
y\in H^1(-\infty,R)\cap C^2(-\infty,R)
\) solves
\begin{equation}\label{eq:L1-linear-ode}
	y''-\eta^2y=q(x)y+h(x)
	\quad\text{on }(-\infty,R),
\end{equation}
and assume that there exists \(\sigma>0\) such that
\[
h(x)=O(e^{(\eta+\sigma)x})\quad \text{as }x\to-\infty.
\]
Then there exists a constant \(a\in\mathbb R\) such that
\[
y(x)=a e^{\eta x}+o(e^{\eta x}),
\ \
y'(x)=\eta a e^{\eta x}+o(e^{\eta x})\quad \text{as }x\to-\infty.
\]
Moreover, if \(h\equiv0\) and \(a=0\), then
\[
y\equiv0
\quad\text{on }(-\infty,R).
\]
\end{lemma}

Lemma \ref{lem:L1-asymptotic} is applied separately to each component, in order to extract
its leading coefficient at \(-\infty\). More precisely, for an
\(H^1\)-solution of
\[
y''-\eta^2y=q(x)y+h(x),
\ \
q\in L^1(-\infty,R),
\ \
h(x)=O(e^{(\eta+\sigma)x})\quad \text{as }\ x\to-\infty,
\]
Lemma \ref{lem:L1-asymptotic} gives a well-defined leading coefficient \(a\) in the
asymptotic expansion
\[
y(x)=a e^{\eta x}+o(e^{\eta x}),
\ \
y'(x)=\eta a e^{\eta x}+o(e^{\eta x})
\quad \text{as }\ x\to-\infty.
\]
In the homogeneous case \(h\equiv0\), this coefficient can vanish only
for the trivial solution. This is the basic tool of extracting the
left asymptotic coefficients \(a_i^-\) of the components \(u_i\).

\vspace {.05cm}

\mbox{\bf Proof of Lemma \ref{lem:L1-asymptotic}.}
The proof is divided into three steps. Firstly, we construct a fundamental system
\(\{\varphi,\psi\}\) for the homogeneous equation
\begin{equation}\label{eq:homoequ}
w''-\eta^2w=q(x)w,
\end{equation}
such that \(\varphi\) has the decaying asymptotic behavior \(e^{\eta x}\), while
\(\psi\) has the growing asymptotic behavior \(e^{-\eta x}\) as
\(x\to-\infty\). Secondly, we use variation of constants to write
\[
y=\alpha\varphi+\beta\psi+p,
\]
where \(p\) is a particular solution produced by \(h\). Finally, the decay
assumption on \(h\) implies that \(p\) and \(p'\) are both of lower order than
\(e^{\eta x}\). Therefore, the condition \(y\in H^1(-\infty,R)\) excludes the
growing mode \(\psi\) and the desired asymptotic estimate is obtained.
	
\medskip
\noindent{\em Step 1:  Construction of two fundamental solutions for the homogeneous equation \eqref{eq:homoequ}.}

Choose \(R_0<R\) sufficiently negative so that
\begin{equation}\label{eq:q-small-tail}
	\frac{1}{2\eta}\|q\|_{L^1(-\infty,R_0)}<1 .
\end{equation}
We first construct a decaying solution of the homogeneous equation
\eqref{eq:homoequ} in $(-\infty,R_0)$. To obtain a solution with the leading behavior
\(e^{\eta x}\) as \(x\to-\infty\), we write
\(
\varphi(x)=e^{\eta x}u(x).
\)
Then $\varphi$ solves \eqref{eq:homoequ} with the  leading behavior
\(e^{\eta x}\) as \(x\to-\infty\), if
\begin{equation}\label{diffequ}
u''+2\eta u'=q(x)u,
\ \
u(x)\to1,\quad u'(x)\to0
\quad\text{as }x\to-\infty .	
\end{equation}
For a bounded continuous solution \(u\), $u$ solves \eqref{diffequ} if
$u$ satisfies the following integral equation:
 \begin{equation}\label{eq:volterra-u}
u(x) = 1+ \frac{1}{2\eta} \int_{-\infty}^{x} \bigl(1-e^{-2\eta(x-s)}\bigr)q(s)u(s)\,ds, \ \  x<R_0 . \end{equation}
Indeed, if \(u\) satisfies \eqref{eq:volterra-u}, then \(qu\in L^1(-\infty,R_0)\), and hence the integral term tends to zero as \(x\to-\infty\).
Thus, \(u(x)\to1\) as \(x\to-\infty\).
Differentiating \eqref{eq:volterra-u} gives that
\[ u'(x) = \int_{-\infty}^{x} e^{-2\eta(x-s)}q(s)u(s)\,ds, \]
which implies that \(u'(x)\to0\) as \(x\to-\infty\). Differentiating once more yields that \[ u''=q(x)u-2\eta u', \] and hence \(u''+2\eta u'=q(x)u\).

We next prove the existence of a solution to the integral equation
\eqref{eq:volterra-u}. Set \(X=C_b((-\infty,R_0])\), equipped with the supremum norm. The
right-hand side of \eqref{eq:volterra-u} defines a map \(T:X\to X\).
Moreover, for any \(u_1,u_2\in X\),
\[
\|Tu_1-Tu_2\|_{L^\infty(-\infty,R_0)}
\le
\frac{1}{2\eta}\|q\|_{L^1(-\infty,R_0)}
\|u_1-u_2\|_{L^\infty(-\infty,R_0)} .
\]
By \eqref{eq:q-small-tail}, \(T\) is a contraction, which implies that
\eqref{eq:volterra-u} has a unique bounded continuous solution
\(u\in C_b((-\infty,R_0])\).
By the argument following \eqref{eq:volterra-u}, this solution \(u\) satisfies \eqref{diffequ}. Applying the standard one-dimensional regularity of ODE, we have
\(u\in C^2(-\infty,R_0)\).
Thus,
\(\varphi=e^{\eta x}u\) is a \(C^2\)-solution of the homogeneous equation
\eqref{eq:homoequ} and satisfies
\begin{equation}\label{eq:phi-asymptotic}
	\varphi(x)=e^{\eta x}(1+o(1)),
	\ \
	\varphi'(x)=\eta e^{\eta x}(1+o(1))\quad \text{as }x\to-\infty.
\end{equation}
In view of \eqref{eq:phi-asymptotic}, by decreasing \(R_0\) further if necessary,
we may assume that
\[
\varphi(x)\neq0
\quad\text{on }(-\infty,R_0).
\]

We next construct the growing homogeneous solution by reduction of order:
\begin{equation}\label{eq:psi-definition}
	\psi(x)
	:=
	2\eta\,\varphi(x)\int_x^{R_0}\frac{dt}{\varphi(t)^2},
	\ \  x<R_0 .
\end{equation}
Then \(\psi\) also solves the homogeneous equation. Moreover,
\[
W(\varphi,\psi)
:=
\varphi\psi'-\varphi'\psi
=
-2\eta .
\]
Thus \(\varphi\) and \(\psi\) form a fundamental system on
\((-\infty,R_0)\). Note from \eqref{eq:phi-asymptotic} that
\[
\frac{1}{\varphi(t)^2}
=
e^{-2\eta t}(1+o(1))\quad \text{as }t\to-\infty.
\]
Hence
\[
\int_x^{R_0}\frac{dt}{\varphi(t)^2}
=
\frac{1}{2\eta}e^{-2\eta x}(1+o(1))\quad \text{as }x\to-\infty.
\]
It follows from \eqref{eq:psi-definition} that
\[
\psi(x)=e^{-\eta x}(1+o(1))\quad \text{as }x\to-\infty.
\]
Differentiating \eqref{eq:psi-definition}, we get that
\[
\psi'
=
2\eta\varphi'\int_x^{R_0}\frac{dt}{\varphi(t)^2}
-
\frac{2\eta}{\varphi}.
\]
Using \eqref{eq:phi-asymptotic} again, we obtain that
\begin{equation}\label{eq:psi-asymptotic}
	\psi(x)=e^{-\eta x}(1+o(1)),
	\ \
	\psi'(x)=-\eta e^{-\eta x}(1+o(1))\quad \text{as }x\to-\infty.
\end{equation}

\medskip
\noindent{\em Step 2: A variation-of-constants representation for \eqref{eq:L1-linear-ode}.}
	
We now treat the inhomogeneous term. Define
\[
I_1(x):=\int_{-\infty}^{x}\psi(s)h(s)\,ds,
\ \
I_2(x):=\int_{-\infty}^{x}\varphi(s)h(s)\,ds .
\]
These integrals are well defined. Indeed, by \eqref{eq:phi-asymptotic},
\eqref{eq:psi-asymptotic}, and the assumption on \(h\),
\[
\psi(s)h(s)=O(e^{\sigma s}),
\ \
\varphi(s)h(s)=O(e^{(2\eta+\sigma)s})\quad \text{as }s\to-\infty.
\]
Set
\begin{equation}\label{eq:p-definition}
	p(x)
	:=
	\frac{1}{2\eta}\varphi(x)I_1(x)
	-
	\frac{1}{2\eta}\psi(x)I_2(x).
\end{equation}
Since
\[
I_1'=\psi h,
\ \
I_2'=\varphi h,
\ \
W(\varphi,\psi)=-2\eta,
\]
a direct computation gives that
\[
p''-(\eta^2+q(x))p=h(x).
\]
Thus, \(p\) is a particular solution of \eqref{eq:L1-linear-ode}. Hence,
\(y-p\) solves the homogeneous equation, and there exist constants
\(\alpha,\beta\in\mathbb R\) such that
\begin{equation}\label{eq:y-decomposition}
	y(x)=\alpha\varphi(x)+\beta\psi(x)+p(x),
	\ \  x<R_0 .
\end{equation}

\medskip
\noindent{\em Step 3: The asymptotic estimates for \(H^1(-\infty,R)\) solutions.}

The above estimates yield that
\[
I_1(x)=O(e^{\sigma x}),
\ \
I_2(x)=O(e^{(2\eta+\sigma)x})\quad \text{as }x\to-\infty,
\]
which imply that
\[
\varphi(x)I_1(x)=O(e^{(\eta+\sigma)x}),
\ \
\psi(x)I_2(x)=O(e^{(\eta+\sigma)x})\quad \text{as }x\to-\infty.
\]
Consequently,
\begin{equation}\label{eq:p-estimate}
	p(x)=O(e^{(\eta+\sigma)x})=o(e^{\eta x})\quad \text{as }x\to-\infty.
\end{equation}
Differentiating \eqref{eq:p-definition}, since  the two terms involving \(h(x)\)
cancel,  we have
\[
p'(x)
=
\frac{1}{2\eta}\varphi'(x)I_1(x)
-
\frac{1}{2\eta}\psi'(x)I_2(x).
\]
By the estimates
of \(I_1,I_2\), we get from   \eqref{eq:phi-asymptotic} and \eqref{eq:psi-asymptotic} that
\begin{equation}\label{eq:p-prime-estimate}
	p'(x)=O(e^{(\eta+\sigma)x})=o(e^{\eta x})\quad \text{as }x\to-\infty.
\end{equation}

Applying \eqref{eq:phi-asymptotic}, \eqref{eq:psi-asymptotic}, and
\eqref{eq:p-estimate}, we get from \eqref{eq:y-decomposition} that
\[
y(x)
=
\alpha e^{\eta x}(1+o(1))
+
\beta e^{-\eta x}(1+o(1))
+
o(e^{\eta x})\quad \text{as }\ x\to-\infty.
\]
Since \(y\in H^1(-\infty,R)\), we have \(y\in L^2(-\infty,R)\).
The function \(e^{-\eta x}\) is not in \(L^2(-\infty,R_0)\), whereas
\[
\varphi\in L^2(-\infty,R_0),
\ \
p=o(e^{\eta x})\quad \text{as }\ x\to-\infty.
\]
Thus \(\beta=0\) and
\[
y(x)=\alpha e^{\eta x}+o(e^{\eta x})\quad \text{as }\ x\to-\infty.
\]
Applying \eqref{eq:y-decomposition} with \(\beta=0\), we obtain from
\eqref{eq:phi-asymptotic} and \eqref{eq:p-prime-estimate} that
\[
y'(x)
=
\alpha\varphi'(x)+p'(x)
=
\eta\alpha e^{\eta x}+o(e^{\eta x})\quad \text{as }\ x\to-\infty.
\]
Taking \(a:=\alpha\), we obtain the desired asymptotics.

It remains  to prove the last assertion in the homogeneous case. Assume
\(h\equiv0\), so that
\[
I_1=I_2\equiv0,
\ \
p\equiv0.
\]
Hence
\[
y(x)=\alpha\varphi(x)+\beta\psi(x),
\ \  x<R_0 .
\]
As before, it follows from \(y\in L^2(-\infty,R_0)\)   that \(\beta=0\). Additionally, if
\(a=0\), then \(\alpha=0\), and hence
\[
y\equiv0
\quad\text{on }(-\infty,R_0).
\]
Since \(y\) solves a second-order linear ordinary differential equation with
continuous coefficients, the uniqueness of Cauchy problem implies that
\[
y\equiv0
\quad\text{on }(-\infty,R).
\]
The proof is complete.
\qed

Applying the preceding scalar lemma to each component, we obtain the following left asymptotic behavior for \(H^1(\mathbb R)^N\)-solutions of \eqref{FM}.

\begin{lemma}
\label{lem:left-asymptotic}
For \(\mu_i=-\eta_i^2<0\), \(i=1,\cdots,N\), assume
$u=(u_1,\cdots,u_N)\in H^1(\mathbb R)^N$
solves
\[
u_i''+2\left(\sum_{j=1}^N u_j^2\right)u_i
=
\eta_i^2u_i,
\ \  i=1,\cdots,N.
\]
Then for each \(i=1,\cdots,N\), there exists a constant
\(a_i^-\in\mathbb R\) such that
\[
u_i(x)=a_i^- e^{\eta_i x}+o(e^{\eta_i x}),\quad
u_i'(x)=\eta_i a_i^- e^{\eta_i x}+o(e^{\eta_i x})\ \  x\to-\infty.
\]
Moreover, if \(u_i\not\equiv0\), then
$a_i^-\neq0$.

\end{lemma}

\begin{proof}
Since \(u\in H^1(\mathbb R)^N\), we have
$2\sum_{j=1}^N u_j^2(x)\in L^1(\mathbb R)$.
For each fixed \(i\), the equation of \(u_i\) can be rewritten as
\[
u_i''-\eta_i^2u_i=q_i(x)u_i,
\ \
q_i(x):=-2\sum_{j=1}^N u_j^2(x)\in C(\mathbb R)\cap L^1(\mathbb R).
\]
By standard one-dimensional regularity of ODE system, we have
\(u\in C^2(\mathbb R)^N\). Applying Lemma~\ref{lem:L1-asymptotic} to \(y=u_i\),
\(\eta=\eta_i\) and \(q=q_i\), we obtain that
\[
u_i(x)=a_i^- e^{\eta_i x}+o(e^{\eta_i x}),\quad
u_i'(x)=\eta_i a_i^- e^{\eta_i x}+o(e^{\eta_i x}),\ \  x\to-\infty.
\]
If \(a_i^-=0\), then the last assertion of
Lemma~\ref{lem:L1-asymptotic} gives that
\[
u_i\equiv0
\quad\text{on }(-\infty,R)
\]
holds
for any fixed \(R\in\mathbb R\). Hence, \(u_i\equiv0\) on
\(\mathbb R\). Therefore, if \(u_i\not\equiv0\), it must have
$a_i^-\neq0$.
\end{proof}

If a
solution of a diagonal second-order system is smaller than its natural
decaying modes at \(-\infty\), and if the source term is quadratically
small on the left, then the following left-uniqueness principle further gives that the solution must vanish on a sufficiently far
left half-line.

\begin{lemma}
\label{lem:zero-left-principle}
Let \(\eta_i>0\), and
set $\eta_*:=\min_{1\le i\le N}\eta_i$,  \(i=1,\cdots,N\).
Assume
\[
z=(z_1,\cdots,z_N)\in H^1(-\infty,R_0)^N
\cap C^2(-\infty,R_0)^N
\]
satisfies  for \(i=1,\cdots,N\),
\[
z_i''-\eta_i^2z_i=F_i(x)
\quad\text{in }(-\infty,R_0),
\ \
z_i(x)=o(e^{\eta_i x})\quad\text{as }x\to-\infty,
\]
where \(F_i\in C((-\infty,R_0))\) denotes the right-hand side associated
with  \(z\). If there exists a constant \(C>0\) such that for  any \(x<R_0\),
\begin{equation}\label{Fibound}
|F_i(x)|
\le
C e^{(\eta_i+2\eta_*)x}
\max_{1\le j\le N} e^{-\eta_jx}|z_j(x)|,
\ \  i=1,\cdots,N,
\end{equation}
then there exists $R<R_0$ such that \(z\equiv0\) on  $(-\infty,R)$.
\end{lemma}

\begin{proof}
Define for \(R<R_0\),
\[
W_R:=
\max_{1\le j\le N}\sup_{x\le R}
e^{-\eta_jx}|z_j(x)|.
\]
By the assumption \(z_j(x)=o(e^{\eta_jx})\) as \(x\to-\infty\),
we obtain that \(W_R<\infty\).
Fix \(R<R_0\) and define
\[
Y_i(x)
:=
\frac1{\eta_i}
\int_{-\infty}^{x}
\sinh\bigl(\eta_i(x-t)\bigr)F_i(t)\,dt ,
\ \  x\le R.
\]
By the definition of \(W_R\), we obtain from \eqref{Fibound} that
\begin{equation}\label{Fibound2}
|F_i(x)|
\le
CW_R e^{(\eta_i+2\eta_*)x},\,\ x\le R.
\end{equation}
Hence the above integral
is well-defined for \(x\le R\). Moreover, differentiating the above
integration gives that
\[
Y_i''-\eta_i^2Y_i=F_i .
\]
Thus,
\[
h_i:=z_i-Y_i
\]
satisfies the homogeneous equation
\[
h_i''-\eta_i^2h_i=0
\quad\text{on }\ (-\infty,R].
\]
Hence,
\[
h_i(x)=\alpha_i e^{\eta_i x}+\beta_i e^{-\eta_i x}.
\]
Since \(z_i,Y_i\in H^1(-\infty,R)\), we have
\(h_i\in H^1(-\infty,R)\). This excludes the mode \(e^{-\eta_i x}\), and hence \(\beta_i=0\).
On the other hand, we obtain from  \eqref{Fibound2} that
\[
\begin{aligned}
|Y_i(x)|
&\le
\frac{CW_R}{2\eta_i}
\int_{-\infty}^{x}
e^{\eta_i(x-t)}e^{(\eta_i+2\eta_*)t}\,dt  \\
&\le
C W_R e^{\eta_i x}e^{2\eta_*x}
=o(e^{\eta_i x})\quad \text{as }x\to-\infty.
\end{aligned}
\]
By the assumption
\(
z_i(x)=o(e^{\eta_i x})\)
as \(x\to-\infty\),
we then have
\(
h_i(x)=o(e^{\eta_i x})
\) as $x\to-\infty$.
Since \(h_i(x)=\alpha_i e^{\eta_i x}\), we get that \(\alpha_i=0\).

We obtain from above that
\[
z_i(x)=Y_i(x)
=
\frac1{\eta_i}
\int_{-\infty}^{x}
\sinh\bigl(\eta_i(x-t)\bigr)F_i(t)\,dt ,
\ \  x\le R.
\]
Since
\[
\sinh\bigl(\eta_i(x-t)\bigr)
\le
\frac12 e^{\eta_i(x-t)},
\ \  t\le x,
\]
we obtain from \eqref{Fibound2} that
\[
\begin{aligned}
	|z_i(x)|
	&\le
	\frac{C}{2\eta_i}W_R
	\int_{-\infty}^{x}
	e^{\eta_i(x-t)}e^{(\eta_i+2\eta_*)t}\,dt  \\
	&=
	\frac{C}{2\eta_i}W_R e^{\eta_i x}
	\int_{-\infty}^{x}e^{2\eta_*t}\,dt  \\
	&=
	\frac{C}{4\eta_i\eta_*}
	W_R e^{\eta_i x}e^{2\eta_*x}.
\end{aligned}
\]
Since \(\eta_i\ge\eta_*\), there exists a constant \(\Theta:=\frac{C}{4\eta_*^2}>0\),
independent of \(R\), such that
\[
e^{-\eta_i x}|z_i(x)|
\le
\Theta W_R e^{2\eta_*x}
\le
\Theta W_R e^{2\eta_*R},
\ \  x\le R.
\]
Taking the supremum over \(x\le R\) and then the maximum over \(i\),
we get that
\[
W_R\le \Theta e^{2\eta_*R}W_R.
\]
Choosing \(R\ll-1\) such that
\(
\Theta e^{2\eta_*R}<1,
\)
we thus conclude that \(W_R=0\). Hence
\[
z_i(x)=0,
\ \  x\le R,\quad i=1,\cdots,N,
\]
which completes the proof.
\end{proof}

\subsection{Classification in the distinct-parameter case}
In this subsection, we prove   Theorem~\ref{thm1.1} on the classification of solutions for \eqref{FM} in the distinct-parameter case.  For any
\(H^1\)-solution \(u\in(H^1(\R)\setminus\{0\})^N\), roughly speaking, the previous lemmas give its leading coefficients at
\(-\infty\).  We shall use these coefficients as the parameters of the
determinant solution constructed in Section~2, and then apply the
left-uniqueness principle to identify the two solutions.

\begin{proof}[\normalfont\bfseries Proof of  Theorem~\ref{thm1.1}]
Set
\[
\eta_i:=\sqrt{-\mu_i}>0,
\ \  i=1,\cdots,N,
\]
so that the system \eqref{FM} can be rewritten as
\[
u_i''
+
2\left(\sum_{j=1}^N u_j^2\right)u_i
=
\eta_i^2u_i,
\ \  i=1,\cdots,N.
\]
By Lemma~\ref{lem:left-asymptotic}, there exist constants
\(a_i\in\mathbb R\) such that for each \(i=1,\cdots,N\),
\[
u_i(x)=a_i e^{\eta_i x}+o(e^{\eta_i x})\quad \text{as }x\to-\infty.
\]
Since \(u_i\not\equiv0\), the last assertion of
Lemma~\ref{lem:left-asymptotic} gives that
\[
a_i\neq0,\ \  i=1,\cdots,N.
\]
Let \(a=(a_1,\cdots,a_N)\) be the vector of leading coefficients obtained
above, and denote \(U^a=(U_1^a,\cdots,U_N^a)\)   the determinant solution
with parameter \(a\), namely
\[
U_i^a(x)
=
\frac{
	a_i e^{\eta_i x}\det(I+B^{(i)}M(x))
}{
	\det(I+M(x))
},
\ \  i=1,\cdots,N.
\]
By Proposition~\ref{prop:determinant-verification}, \(U^a\) is a
\((H^1(\mathbb R)\setminus\{0\})^N\) solution of the  system \eqref{FM}. Moreover, since
\(M(x)\to0\) as \(x\to-\infty\), we have
\[
U_i^a(x)=a_i e^{\eta_i x}+o(e^{\eta_i x})\quad \text{as }\ x\to-\infty.
\]

It remains to prove that \(u\equiv U^a\). Set
\[
w_i:=u_i-U_i^a,
\ \
S_u:=\sum_{j=1}^N u_j^2,
\ \
S_a:=\sum_{j=1}^N (U_j^a)^2.
\]
Then
\[
w_i''-\eta_i^2w_i
=
-2\left(S_u u_i-S_a U_i^a\right)
=:F_i .
\]
Since \(u\) and \(U^a\) have the same left asymptotic data, we have
\[
w_i(x)=o(e^{\eta_i x})\quad \text{as }x\to-\infty.
\]
Set
\[
\eta_*:=\min_{1\le j\le N}\eta_j .
\]
By the above asymptotics, there exist \(R_0<0\) and \(C_0>0\)
such that for all \(x\le R_0\),
\[
|u_i(x)|+|U_i^a(x)|
\le
C_0 e^{\eta_i x},
\ \  i=1,\cdots,N.
\]
Define
\[
Z(x):=
\max_{1\le j\le N}e^{-\eta_jx}|w_j(x)|, \ \  x\le R_0,
\]
so that
\[
|w_j(x)|\le Z(x)e^{\eta_jx},
\ \  x\le R_0.
\]
Using
\[
S_u u_i-S_a U_i^a
=
S_u w_i+(S_u-S_a)U_i^a,
\ \
S_u-S_a
=
\sum_{j=1}^N (u_j+U_j^a)w_j,
\]
we obtain that
\[
\begin{aligned}
	|F_i(x)|
	&\le
	2S_u(x)|w_i(x)|
	+
	2|U_i^a(x)|
	\sum_{j=1}^N |u_j(x)+U_j^a(x)|\,|w_j(x)|  \\
	&\le
	C e^{(\eta_i+2\eta_*)x}Z(x), \ \  x\le R_0,
\end{aligned}
\]
where \(C>0\) is independent of \(x\). Thus, \(w\) satisfies the
assumptions of Lemma~\ref{lem:zero-left-principle}, which implies that
\[
w\equiv0
\ \ \text{on a sufficiently far left half-line}.
\]
In particular, we have for some \(R<R_0\),
\[
u_i(R)=U_i^a(R),
\ \
u_i'(R)=(U_i^a)'(R),
\ \  i=1,\cdots,N.
\]
By the uniqueness of solutions for ODE systems, we obtain that
\[
u\equiv U^a
\quad\text{on }\ \mathbb R.
\]
This proves Theorem~\ref{thm1.1}.
\end{proof}

\subsection{Nondegeneracy in the distinct-parameter case}
In this subsection, we prove Theorem~\ref{prop:nondegeneracy} on the
nondegeneracy of the determinant family in the distinct-parameter case.
It is clear that the tangent directions obtained by differentiating the
determinant family belong to the kernel. To prove that these are the only \(H^1\)-kernel
directions, we again use the left asymptotics and the left-uniqueness
principle established above.

\begin{proof}[\normalfont\bfseries Proof of Theorem~\ref{prop:nondegeneracy}]
For each \(k=1,\cdots,N\), differentiating the identity
\[
(U_i^a)''
+
2\left(\sum_{j=1}^N (U_j^a)^2\right)U_i^a
=
\eta_i^2U_i^a,
\ \  i=1,\cdots,N,
\]
with respect to \(a_k\), we obtain that
\(
\mathcal L_{U^a}(\partial_{a_k}U^a)=0,
\)
where $\mathcal L_{U^a}$  is given by
\[
(\mathcal L_{U^a}\phi)_i
:=
-\phi_i''
+\eta_i^2\phi_i
-2S^a\phi_i
-4U_i^a\sum_{j=1}^N U_j^a\phi_j,
\ \  i=1,\cdots,N .
\]
Here $S^a=\left(\sum_{j=1}^N (U_j^a)^2\right)$.
Furthermore, direct computations yield that $\partial_{a_k}U^a\in H^1(\R)^N$ holds for $a_k\neq 0$. Thus, we have for every \(a\in(\mathbb R\setminus\{0\})^N\),
\[
\operatorname{span}
\left\{
\partial_{a_1}U^a,\cdots,\partial_{a_N}U^a
\right\}
\subset
\ker_{H^1(\mathbb R)^N}\mathcal L_{U^a} .
\]
Conversely, we have for
$\phi=(\phi_1,\cdots,\phi_N)
\in
\ker_{H^1(\mathbb R)^N}\mathcal L_{U^a}$,
\[
\phi_i''-\eta_i^2\phi_i
=
-2S^a\phi_i
-4U_i^a\sum_{j=1}^N U_j^a\phi_j,
\ \  i=1,\cdots,N.
\]
Since \(U_i^a(x)=O(e^{\eta_i x})\) as \(x\to-\infty\), we obtain that
\(S^a\in L^1(-\infty,R_0)\). Moreover, the one-dimensional Sobolev
embedding gives that
$\phi\in L^\infty(\mathbb R)^N$.

For each \(i\), rewrite the equation as
\[
\phi_i''-\eta_i^2\phi_i
=
-2S^a\phi_i+h_i(x),
\]
where
\(
h_i(x):=-4U_i^a(x)\sum_{j=1}^N U_j^a(x)\phi_j(x).
\)
Set
$
\eta_*:=\min_{1\le j\le N}\eta_j,
$
so that
\[
|h_i(x)|
\le
C |U_i^a(x)|\sum_{j=1}^N |U_j^a(x)|
\le
C e^{(\eta_i+\eta_*)x}\quad \text{as }x\to-\infty.
\]
In particular,
$
h_i(x)=o(e^{\eta_i x})
$ as $x\to-\infty$.
Applying Lemma~\ref{lem:L1-asymptotic}, we obtain that there exists a constant
\(b_i^-\in\mathbb R\) such that
\[
\phi_i(x)=b_i^-e^{\eta_i x}+o(e^{\eta_i x})\quad \text{as }x\to-\infty.
\]
On the other hand, the determinant formula of \(U^a\) gives that
\[
\partial_{a_k}U_i^a(x)
=
\delta_{ik}e^{\eta_i x}+o(e^{\eta_i x})\quad \text{as }\ x\to-\infty.
\]
Define
\[
\psi
:=
\sum_{k=1}^N b_k^-\partial_{a_k}U^a,
\]
so that $\mathcal L_{U^a}\psi=0$
and
\[
\psi_i(x)=b_i^-e^{\eta_i x}+o(e^{\eta_i x})\quad \text{as }\ x\to-\infty.
\]
Hence
\(
z:=\phi-\psi
\)
satisfies
$\mathcal L_{U^a}z=0$,
and
\[
z_i(x)=o(e^{\eta_i x})
\ \ \text{as }\ x\to-\infty, \ \ i=1,\cdots,N.
\]
Moreover, $z(x)$ satisfies
\[
z_i''-\eta_i^2 z_i
=
-2S^a z_i
-4U_i^a\sum_{j=1}^N U_j^a z_j
=:F_i(x).
\]
Set
\[
Z(x):=\max_{1\le j\le N}e^{-\eta_jx}|z_j(x)|, \ \  x\le R_0.
\]
Since \(U_i^a(x)=O(e^{\eta_i x})\) as \(x\to-\infty\), after decreasing
\(R_0\) if necessary, we have
\[
|F_i(x)|
\le
C e^{(\eta_i+2\eta_*)x}Z(x),
\ \  x\le R_0.
\]
Thus, \(z\) satisfies the assumptions of
Lemma~\ref{lem:zero-left-principle}. This gives that there exists \(R<R_0\) such that
\[
z\equiv0 \quad\text{on }\ (-\infty,R).
\]
In particular,
\[
z_i(R)=z_i'(R)=0,\ \  i=1,\cdots,N.
\]
Since \(z\) satisfies the linearized ODE system with continuous
coefficients, the uniqueness theorem of ODE systems
implies that
$z\equiv0$ on $\R$.
Thus,
\(
\phi
=
\sum_{k=1}^N b_k^-\partial_{a_k}U^a,
\)
which proves that
\[
\ker_{H^1(\mathbb R)^N}\mathcal L_{U^a}
\subset
\operatorname{span}
\left\{
\partial_{a_1}U^a,\cdots,\partial_{a_N}U^a
\right\}.
\]
Combining above two inclusions, we obtain that
\[
\ker_{H^1(\mathbb R)^N}\mathcal L_{U^a}
=
\operatorname{span}
\left\{
\partial_{a_1}U^a,\cdots,\partial_{a_N}U^a
\right\}.
\]

Finally, since
\[
\partial_{a_k}U_i^a(x)
=
\delta_{ik}e^{\eta_i x}+o(e^{\eta_i x})\quad \text{as }x\to-\infty,
\]
the vector-valued functions
$
\partial_{a_1}U^a,\cdots,\partial_{a_N}U^a
$
are linearly independent. Hence
$
\dim\ker_{H^1(\mathbb R)^N}\mathcal L_{U^a}=N,
$
which completes the proof of Theorem~\ref{prop:nondegeneracy}.
\end{proof}

\section{Classification and Nondegeneracy with Repeated Spectral Parameters}
\label{sec:multiple-spectral-parameters}
This section is concerned with the classification and nondegeneracy of \(H^1(\mathbb R)^N\)-solutions for \eqref{FM}, in the case where the spectral parameters are allowed to have multiplicities. In Subsection~\ref{sub4.1}, we shall prove  Theorem~\ref{thm:completeness-multiple-root} on the classification of solutions for \eqref{eq:multi-system}. In Subsection~\ref{sub4.2}, we prove the   nondegeneracy of Theorem~\ref{prop:nondegeneracy-multiple-roots}  for the linearized operator \(\mathcal L_{U^{a,n}}\) defined in \eqref{eq:multi-LU}.

Let
\[
\mu_1<\cdots<\mu_K<0
\]
be the distinct spectral parameters, and denote \(k_\alpha\) to be the
multiplicity of \(\mu_\alpha\), so that
\[
k_1+\cdots+k_K=N.
\]
Set
\[
N_0:=0,\ \
N_\alpha:=k_1+\cdots+k_\alpha,
\ \
I_\alpha:=\{N_{\alpha-1}+1,\cdots,N_\alpha\}.
\]
The full list of spectral parameters counted with multiplicity is denoted
by \(\nu_1,\cdots,\nu_N\) and
\[
\nu_i=\mu_\alpha,\ \  i\in I_\alpha.
\]
For $\alpha=1,2,\cdots,K$, we also write
$
\mu_\alpha=-\eta_\alpha^2<0$.

The key observation of   this section lies in the fact that repeated spectral parameters do not produce new
scalar profiles. Instead, the components in the same spectral block are
locked together and differ only by constant coefficients. Hence, each
nontrivial block \(I_\alpha\) reduces to a single scalar profile
\(V_\alpha\), and the collection \(V=(V_1,\cdots,V_K)\) solves the reduced
distinct-parameter system \eqref{eq:reduced-system}. Therefore, the
classification in the repeated-parameter case follows from the
distinct-parameter classification of Theorem~\ref{thm1.1}, together with the
internal rotational freedom inside each block.

\subsection{Classification with repeated spectral parameters}\label{sub4.1}
The purpose of this subsection is to prove
Theorem~\ref{thm:completeness-multiple-root}, namely the complete
classification of \(H^1(\mathbb R)^N\)-solutions for
\eqref{eq:multi-system} with repeated spectral parameters.

\begin{proof}[\normalfont\bfseries Proof of Theorem~\ref{thm:completeness-multiple-root}]
Let
\(
S=\sum_{\ell=1}^N u_\ell^2 ,
\)
and fix \(\alpha\in\{1,\cdots,K\}\). It then yields from \eqref{eq:multi-nu} that for any \(i,j\in I_\alpha\),
\(
\nu_i=\nu_j=\mu_\alpha .
\)
Hence, both \(u_i\) and \(u_j\) solve the same linear equation
\[
y''+2S y+\mu_\alpha y=0.
\]
Define
\[
W_{ij}:=u_i'u_j-u_i u_j',
\]
so that $W_{ij}'
=
u_i''u_j-u_i u_j''
=0$.
Thus, \(W_{ij}\) is constant on \(\mathbb R\). Moreover, since
\(u_i,u_j\in H^1(\mathbb R)\), we have
$
u_i'u_j,\ u_i u_j'\in L^1(\mathbb R),
$
which implies that \(W_{ij}\in L^1(\mathbb R)\). Since a constant function belonging to
\(L^1(\mathbb R)\) must be zero, we have
\[
W_{ij}\equiv0\ \ \hbox{for any }i,j\in I_\alpha.
\]

Since $\sum_{i\in I_\alpha} u_i^2(x)\not\equiv 0$, there exists
\(i_\alpha\in I_\alpha\) such that
$u_{i_\alpha}\not\equiv0$.
For any \(i\in I_\alpha\), choose \(x_\alpha\in\mathbb R\) such that
\(u_{i_\alpha}(x_\alpha)\neq0\), and set
\[
c_i:=\frac{u_i(x_\alpha)}{u_{i_\alpha}(x_\alpha)} .
\]
Define
\[
w_i:=u_i-c_i u_{i_\alpha},
\]
so that \(w_i\) satisfies the same linear equation
\[
w_i''+2S w_i+\mu_\alpha w_i=0 .
\]
Moreover, we have $w_i(x_\alpha)=0$.
Since \(W_{i i_\alpha}\equiv0\), we also have
\[
u_i'(x_\alpha)u_{i_\alpha}(x_\alpha)
-
u_i(x_\alpha)u_{i_\alpha}'(x_\alpha)=0,
\]
which gives that $w_i'(x_\alpha)=0$.
By the uniqueness of   second-order linear ODE, we obtain that
\(
w_i\equiv0.
\)
We thus have
\[
u_i=c_i u_{i_\alpha},
\ \  i\in I_\alpha .
\]

Set
\[
d_\alpha:=\left(\sum_{i\in I_\alpha}c_i^2\right)^{1/2}>0,
\ \
V_\alpha:=d_\alpha u_{i_\alpha},
\ \
n_i^{(\alpha)}:=\frac{c_i}{d_\alpha},
\quad i\in I_\alpha .
\]
Then
\[
n^{(\alpha)}
=
(n_i^{(\alpha)})_{i\in I_\alpha}
\in \mathbb S^{k_\alpha-1},
\]
and
\[
u_i=n_i^{(\alpha)}V_\alpha,
\ \  i\in I_\alpha .
\]
Consequently,
$
\sum_{i\in I_\alpha}u_i^2=V_\alpha^2,
$
and hence
$
S=\sum_{\beta=1}^K V_\beta^2.
$
Substitute $u_i=n_i^{(\alpha)}V_\alpha$
into the equation of \(u_i\), and choose \(i=i_\alpha\) so that
\(n_{i_\alpha}^{(\alpha)}\neq0\). We then obtain that
\[
V_\alpha''
+
2\left(\sum_{\beta=1}^K V_\beta^2\right)V_\alpha
+
\mu_\alpha V_\alpha
=0,
\ \  \alpha=1,\cdots,K.
\]
Thus
$
V=(V_1,\cdots,V_K)
$
is an \(H^1(\mathbb R)^K\) solution of the reduced distinct-root system
\eqref{eq:reduced-system}. Since $\sum_{i\in I_\alpha} u_i^2(x)\not\equiv 0$, we have
\(V_\alpha\not\equiv0\).

Applying Theorem~\ref{thm1.1} to the reduced system
\eqref{eq:reduced-system}, it yields that there exists
$
a=(a_1,\cdots,a_K)\in(\mathbb R\setminus\{0\})^K
$
such that
$
V=V^a .
$
Therefore
\[
u_i=n_i^{(\alpha)}V_\alpha^a,
\ \  i\in I_\alpha,\quad \alpha=1,\cdots,K.
\]
The definition of \(\mathcal M_{\mathrm{multi}}\) in
\eqref{eq:multi-family} then gives that
$
u\in \mathcal M_{\mathrm{multi}},
$
and we are done.
\end{proof}

\subsection{Nondegeneracy  with repeated spectral parameters}\label{sub4.2}
In this subsection, we prove the nondegeneracy of Theorem~\ref{prop:nondegeneracy-multiple-roots}. The argument follows the same reduction principle as in the classification theorem, except that it is now at the linearized level. The component of a kernel element parallel to \(n^{(\alpha)}\) is governed by the reduced \(K\)-component linearized operator, while the orthogonal component inside each block is generated exactly by rotations of \(n^{(\alpha)}\).

\begin{proof}[\normalfont\bfseries Proof of Theorem~\ref{prop:nondegeneracy-multiple-roots}]
Let
\(U^{a,n}\in\mathcal M_{\mathrm{multi}}\) be given by
\[
U_i^{a,n}(x)=n_i^{(\alpha)}V_{\alpha}^a(x),
\ \
i\in I_\alpha,\quad \alpha=1,\cdots,K,
\]
where \(V^a=(V_1^a,\cdots,V_K^a)\) is the reduced solution with distinct
spectral parameters and
\(n^{(\alpha)}=(n_i^{(\alpha)})_{i\in I_\alpha}\in
\mathbb S^{k_\alpha-1}\). Since \(|n^{(\alpha)}|=1\), we have
\[
S^{a}(x):=\sum_{\ell=1}^N (U_\ell^{a,n})^2
=
\sum_{\beta=1}^K (V_\beta^a)^2 .
\]
The linearized operator at \(U^{a,n}\) is
\[
(\mathcal L_{U^{a,n}}\phi)_i
=
-\phi_i''
-\nu_i\phi_i
-2S^{a}\phi_i
-4U_i^{a,n}\sum_{\ell=1}^N U_\ell^{a,n}\phi_\ell,
\ \  i=1,\cdots,N.
\]
We also recall the tangent directions of the manifold
\(\mathcal M_{\mathrm{multi}}\) at \(U^{a,n}\). The directions generated by
the reduced parameters are
\[
\Phi^{(\gamma)}=(\Phi_1^{(\gamma)},\cdots,\Phi_N^{(\gamma)}),\ \ \hbox{where}\,\ \Phi_i^{(\gamma)}
=
n_i^{(\alpha)}\partial_{a_\gamma}V_\alpha^a,
\ \
i\in I_\alpha,\quad \gamma=1,\cdots,K.
\]
The rotational directions inside the block \(I_\alpha\) are
\[
\Psi^{(\alpha,\xi)}=(\Psi_1^{(\alpha,\xi)},\cdots,\Psi_N^{(\alpha,\xi)}),\ \ \hbox{where}\,\ \Psi_i^{(\alpha,\xi)}
=
\begin{cases}
	\xi_i V_\alpha^a, & i\in I_\alpha,\\
	0, & i\notin I_\alpha,
\end{cases}
\ \
\xi\in T_{n^{(\alpha)}}\mathbb S^{k_\alpha-1},
\]
where
\[
T_{n^{(\alpha)}}\mathbb S^{k_\alpha-1}
=
\left\{\xi\in\mathbb R^{k_\alpha}:\ \xi\cdot n^{(\alpha)}=0\right\}.
\]
Thus, \(T_{U^{a,n}}\mathcal M_{\mathrm{multi}}\) is the span of all
\(\Phi^{(\gamma)}\) and \(\Psi^{(\alpha,\xi)}\).

We now prove that
\[
T_{U^{a,n}}\mathcal M_{\mathrm{multi}}
\subset
\ker_{H^1(\mathbb R)^N}\mathcal L_{U^{a,n}} .
\]
Indeed, the directions \(\Phi^{(\gamma)}\) are obtained by differentiating
the family \(U^{a,n}\) with respect to the reduced parameters \(a_\gamma\).
Since \(U^{a,n}\) solves the full \(N\)-component system for all admissible
\(a\), we have
\[
\mathcal L_{U^{a,n}}\Phi^{(\gamma)}=0,
\ \  \gamma=1,\cdots,K.
\]
Similarly, fix $\alpha$ and let
$\xi\in\mathbb R^{k_\alpha}$ satisfy
$\xi\cdot n^{(\alpha)}=0$.
Then there exists a smooth curve
$n_\varepsilon^{(\alpha)}\in\mathbb S^{k_\alpha-1}$
such that
\[
n_0^{(\alpha)}=n^{(\alpha)},
\ \
\frac{d}{d\varepsilon}n_\varepsilon^{(\alpha)}
\Big|_{\varepsilon=0}
=\xi.
\]
Define the perturbed solution by
\[
U_i^\varepsilon
=
\begin{cases}
	n_{\varepsilon,i}^{(\alpha)} V_\alpha^a,
	& i\in I_\alpha,\\
	U_i^{a,n},
	& i\notin I_\alpha.
\end{cases}
\]
Since the system is invariant under orthogonal
rotations inside each block $I_\alpha$,
$U^\varepsilon$ remains a solution for all
$\varepsilon$.
Differentiating the equation with respect to
$\varepsilon$ at $\varepsilon=0$, it
yields that
$
\mathcal L_{U^{a,n}} \Psi^{(\alpha,\xi)}=0.
$
This gives that
\[
T_{U^{a,n}}\mathcal M_{\mathrm{multi}}
\subset
\ker_{H^1(\mathbb R)^N}\mathcal L_{U^{a,n}} .
\]

It remains to prove the reverse inclusion. Let
\[
\phi=(\phi_1,\cdots,\phi_N)
\in
\ker_{H^1(\mathbb R)^N}\mathcal L_{U^{a,n}} .
\]
For each block \(I_\alpha\), decompose \(\phi\) into its component parallel
to \(n^{(\alpha)}\) and its orthogonal component. Define
\[
p_\alpha
:=
\sum_{i\in I_\alpha}n_i^{(\alpha)}\phi_i,\quad r_i
:=
\phi_i-n_i^{(\alpha)}p_\alpha,
\ \  i\in I_\alpha,\]
then
$
\phi_i=n_i^{(\alpha)}p_\alpha+r_i
$ for $i\in I_\alpha$,
and
$
\sum_{i\in I_\alpha}n_i^{(\alpha)}r_i=0.
$
Multiplying the \(i\)-th equation by \(n_i^{(\alpha)}\) and summing over
\(i\in I_\alpha\), we obtain that
\[
-\left(\sum_{i\in I_\alpha}n_i^{(\alpha)}\phi_i\right)''
-\mu_\alpha \sum_{i\in I_\alpha}n_i^{(\alpha)}\phi_i
-2S^a\sum_{i\in I_\alpha}n_i^{(\alpha)}\phi_i
-4\left(\sum_{\ell=1}^N U_\ell^{a,n}\phi_\ell\right)
\sum_{i\in I_\alpha}n_i^{(\alpha)}U_i^{a,n}
=0.
\]
Since
\(
\sum_{i\in I_\alpha}n_i^{(\alpha)}\phi_i=p_\alpha
\),
\[
\sum_{i\in I_\alpha}n_i^{(\alpha)}U_i^{a,n}
=
V_\alpha^a\sum_{i\in I_\alpha}(n_i^{(\alpha)})^2
=
V_\alpha^a,
\]
and
\[
\sum_{\ell=1}^N U_\ell^{a,n}\phi_\ell
=
\sum_{\beta=1}^K V_\beta^a p_\beta,
\]
we get that
\[
-p_\alpha''
-\mu_\alpha p_\alpha
-2S^{a}p_\alpha
-4V_\alpha^a\sum_{\beta=1}^K V_\beta^a p_\beta
=0.
\]
Hence, \(p=(p_1,\cdots,p_K)\) belongs to
\(\ker_{H^1(\mathbb R)^K}\mathcal L_{V^a}\).
By the nondegeneracy of the reduced solution in Theorem \ref{prop:nondegeneracy}, we have
\[
\ker_{H^1(\mathbb R)^K}\mathcal L_{V^a}
=
\operatorname{span}
\left\{
\partial_{a_1}V^a,\cdots,\partial_{a_K}V^a
\right\},
\]
and hence there exist constants \(c_1,\cdots,c_K\) such that
\(
p=\sum_{\gamma=1}^K c_\gamma \partial_{a_\gamma}V^a.
\)

We now identify the orthogonal part. Recall that \(r_i=\phi_i-n_i^{(\alpha)}p_\alpha\). For \(i\in I_\alpha\), subtract
\(n_i^{(\alpha)}\) times the equation of \(p_\alpha\) from the \(i\)-th
equation of \(\phi_i\), and use
\[
U_i^{a,n}=n_i^{(\alpha)}V_\alpha^a,
\ \
\sum_{\ell=1}^N U_\ell^{a,n}\phi_\ell
=
\sum_{\beta=1}^K V_\beta^a p_\beta.
\]
It yields that
\[
-r_i''
-\mu_\alpha r_i
-2S^{a} r_i
=0,
\ \  i\in I_\alpha .
\]
Therefore, both \(r_i\) and \(V_\alpha^a\) solve the same linear equation
\[
-y''
-\mu_\alpha y
-2S^a y=0,
\]
their Wronskian
\[
W_i^{(\alpha)}
:=
r_i(V_\alpha^a)'-r_i'V_\alpha^a
\]
is constant, due to the fact that $(W_i^{(\alpha)})'=0$. Since \(r_i,V_\alpha^a\in H^1(\mathbb R)\), we have
$
r_i(V^a_\alpha)',\ r_i'V_\alpha^a\in L^1(\mathbb R).
$
Hence \(W_i^{(\alpha)}\in L^1(\mathbb R)\), which gives that
\(
W_i^{(\alpha)}\equiv0.
\)
Since \(V^a_\alpha\not\equiv0\), choose \(x_0\in\mathbb R\) such that
$
V^a_\alpha(x_0)\neq0.
$
Set
\[
d_i^{(\alpha)}:=\frac{r_i(x_0)}{V^a_\alpha(x_0)},
\]
so that
$
r_i(x_0)=d_i^{(\alpha)}V^a_\alpha(x_0).
$
Moreover, since $W_i^{(\alpha)}\equiv0$ that  \(W_i^{(\alpha)}(x_0)=0\), we have
$
r_i'(x_0)=d_i^{(\alpha)}(V^a_\alpha)'(x_0).
$
Since \(r_i\) and \(d_i^{(\alpha)}V^a_\alpha\) solve the same second-order
linear equation,  the uniqueness gives that
\[
r_i=d_i^{(\alpha)}V^a_\alpha,
\ \  i\in I_\alpha.
\]
The orthogonality condition thus gives that
\[
\sum_{i\in I_\alpha}n_i^{(\alpha)}r_i=0,
\]
which therefore implies that
$
\sum_{i\in I_\alpha}n_i^{(\alpha)}d_i^{(\alpha)}=0,
$
i.e.,
$
d^{(\alpha)}\cdot n^{(\alpha)}=0.
$

By the parallel and orthogonal parts, we now obtain that for \(i\in I_\alpha\),
\[
\phi_i
=
n_i^{(\alpha)}
\sum_{\gamma=1}^K c_\gamma \partial_{a_\gamma}V^a_\alpha
+
d_i^{(\alpha)}V^a_\alpha,
\]
where
$
d^{(\alpha)}\cdot n^{(\alpha)}=0.
$
Thus, \(\phi\) is a linear combination of the parameter directions
\(\Phi^{(\gamma)}\) and the rotation directions
\(\Psi^{(\alpha,\xi)}\). Hence
\[
\ker_{H^1(\mathbb R)^N}\mathcal L_{U^{a,n}}
\subset
T_{U^{a,n}}\mathcal M_{\mathrm{multi}}.
\]
Together with the first inclusion, this proves that
\[
\ker_{H^1(\mathbb R)^N}\mathcal L_{U^{a,n}}
=
T_{U^{a,n}}\mathcal M_{\mathrm{multi}}.
\]

Finally, the parameter directions are independent, due to the nondegeneracy
of the reduced solution and the local parametrization \(a\mapsto V^a\).
The rotation directions in the block \(I_\alpha\) form the tangent space
\(T_{n^{(\alpha)}}\mathbb S^{k_\alpha-1}\), and hence contribute to
\(k_\alpha-1\) independent directions. These two types of directions are
independent, because they correspond respectively to the parallel and
orthogonal components in each block. We thus have
\[
\dim T_{U^{a,n}}\mathcal M_{\mathrm{multi}}
=
K+\sum_{\alpha=1}^K(k_\alpha-1)
=
N,
\]
and the proof is therefore complete.
\end{proof}

\section{\(L^2\)-Mass Formula of Solutions}
In this section, we shall prove Theorem~\ref{thm L2iden} in two steps. In Subsection 5.1 we first derive a pointwise algebraic identity of the components belonging to the same spectral block. We then use this identity to establish the exact \(L^2\)-mass formula of the classified solutions. The derivation of the algebraic identity relies on the constants of motion. After rewriting these constants in a polynomial form, we then evaluate the resulting polynomial identity at suitable spectral values. This gives the desired identity in the distinct-parameter case, while the repeated-parameter case follows by reducing each block of equal parameters to a single effective component.

\subsection{An Algebraic identity}
In this subsection we first prove the following algebraic identity, which is used in the next subsection to derive the exact \(L^2\)-mass formula.

\begin{proposition}\label{thm:algiden}
Under the same assumptions of
Theorem~\ref{thm:completeness-multiple-root}, we have for any \(\alpha=1,\cdots,K\),
\begin{equation}\label{eq:algebraic-identity}
	\sum_{j\in I_\alpha}
	\left[
	(u_j')^2+\big(\sum_{l=1}^N u_l^2+\mu_\alpha\big)u_j^2
	\right]
	+
	\sum_{\beta\neq\alpha}
	\sum_{\substack{j\in I_\alpha\\ l\in I_\beta}}
	\frac{(u_j'u_l-u_ju_l')^2}{\mu_\beta-\mu_\alpha}
	=0 .
\end{equation}
\end{proposition}

We remark that the constants of motion for \eqref{FM} were established in
\cite{GLW2026}. For convenience, we record
them in their generating polynomial form. More precisely, when the left-hand
side of \eqref{eq:polynomial-identity-function} below is expanded as a
polynomial in \(\lambda\), it has degree at most \(N-1\), and for each
\(k=1,\cdots,N\), the coefficient of \(\lambda^{N-k}\) is exactly the
\(k\)-th constant of motion in \cite[Lemma~5.1]{GLW2026}. In order to prove Proposition \ref{thm:algiden},
we first prove the following constants of motion in a simplified approach.

\begin{lemma}\label{lem:polynomial-identity-function}
For $\mu_1\leq \mu_2\leq\cdots\leq \mu_N<0$, let \(u=(u_1,\cdots,u_N)\in H^1(\mathbb R)^N\) be a solution of
\eqref{FM}, namely
\[
-u_j''-2Su_j=\mu_j u_j,\ \
S=\sum_{i=1}^N u_i^2 ,
\]
and set
$
W_{jl}:=u_j'u_l-u_j u_l'$ for $1\leq j,l\leq N$. Then the following polynomial identity holds
	\begin{equation}\label{eq:polynomial-identity-function}
		\begin{split}
			\sum_{j=1}^N
			\left[
			(u_j')^2+(S+\mu_j)u_j^2
			\right]
			\prod_{i\neq j}(\lambda+\mu_i)
			+
			\sum_{1\leq j<l\leq N}
			W_{jl}^2
			\prod_{i\neq j,l}(\lambda+\mu_i)
			\equiv 0 \ \ \text{in} \ \ \R
		\end{split}
	\end{equation}
	as a polynomial identity in \(\lambda\).
\end{lemma}

\begin{proof}
Set for \(j=1,\cdots,N\),
\[
P_j:=(u_j')^2+(S+\mu_j)u_j^2 ,
\]
and define for
\(\lambda\notin\{-\mu_1,\cdots,-\mu_N\}\),
\[
\mathcal C(\lambda,x)
:=
\sum_{j=1}^N
\frac{P_j}{\lambda+\mu_j}
+
\sum_{1\leq j<l\leq N}
\frac{W_{jl}^2}{(\lambda+\mu_j)(\lambda+\mu_l)} .
\]
We claim that \(\mathcal C(\lambda,x)\) is independent of \(x\).

In order to prove the above claim, it   suffices  to show that $	\partial_x \mathcal C(\lambda,x)\equiv 0$.  Note that
\[
\begin{split}
	\partial_x \mathcal C(\lambda,x)
	&=
	\sum_{j=1}^N
	\frac{P_j'}{\lambda+\mu_j}
	+
	\sum_{1\leq j<l\leq N}
	\frac{2W_{jl}W_{jl}'}{(\lambda+\mu_j)(\lambda+\mu_l)} .
\end{split}
\]
We next compute $P_j'$ as follows. Since
\(
u_j''=-(2S+\mu_j)u_j,
\)
we obtain that
\[
\begin{split}
	P_j'=
	2u_j'u_j''+S'u_j^2+2(S+\mu_j)u_ju_j'=
	S'u_j^2-2Su_ju_j' .
\end{split}
\]
Because \(S'=2\sum_{m=1}^N u_m u_m'\), we have
\[
P_j'
=
-2\sum_{m=1}^N u_j u_m W_{jm},
\]
On the other hand, we obtain that
\[
W_{jl}'
=
u_j''u_l-u_ju_l''
=
(\mu_l-\mu_j)u_ju_l .
\]
For each pair \(j<l\), the coefficient of $u_ju_lW_{jl}$ coming from $\sum_{j=1}^N
\frac{P_j'}{\lambda+\mu_j}
$ satisfies
\[
-\,\frac{2u_j u_l W_{jl}}{\lambda+\mu_j}
+
\frac{2u_j u_l W_{jl}}{\lambda+\mu_l}
=
\frac{2(\mu_j-\mu_l)u_j u_l W_{jl}}
{(\lambda+\mu_j)(\lambda+\mu_l)} .
\]
The contribution coming from $
\sum_{1\leq j<l\leq N}
\frac{2W_{jl}W_{jl}'}{(\lambda+\mu_j)(\lambda+\mu_l)}$ also satisfies
\[
\frac{2(\mu_l-\mu_j)u_j u_l W_{jl}}
{(\lambda+\mu_j)(\lambda+\mu_l)} .
\]
Hence,
$\partial_x \mathcal C(\lambda,x)=0$, and  it proves the claim that  $
\mathcal C(\lambda,x)$ is a constant.

On the other hand, since $u\in H^1(\R)^N$, we have $\mathcal C(\lambda,x)\in L^1(\R)$. Since $\mathcal C(\lambda,x)$ is a constant in \(x\), it must have
\[
\mathcal C(\lambda,x)\equiv0 .
\]
Multiplying this identity by
\(
\prod_{i=1}^N(\lambda+\mu_i)
\), it
gives  that \eqref{eq:polynomial-identity-function} holds for every
\(\lambda\notin\{-\mu_1,\cdots,-\mu_N\}\). Since the left-hand side is a
polynomial in \(\lambda\), the identity \eqref{eq:polynomial-identity-function} holds for all \(\lambda\).
\end{proof}

\begin{lemma}\label{lem:algebraic-identity-function}
Under the assumptions of Lemma~\ref{lem:polynomial-identity-function},  assume
\[
\mu_1<\mu_2<\cdots<\mu_N<0.
\]
Then we have for each \(j=1,2,\cdots,N\),
\begin{equation}\label{eq:algebraic-identity-function}
	(u_j')^2+(S+\mu_j)u_j^2
	+
	\sum_{l\neq j}
	\frac{(u_j'u_l-u_ju_l')^2}{\mu_l-\mu_j}
	=0
	\ \ \text{in }\ \mathbb R.
\end{equation}
\end{lemma}

\begin{proof}
By Lemma \ref{lem:polynomial-identity-function}, we have
\[
\begin{split}
	\sum_{r=1}^N
	\left[
	(u_r')^2+(S+\mu_r)u_r^2
	\right]
	\prod_{i\neq r}(\lambda+\mu_i)
	+
	\sum_{1\leq r<s\leq N}
	W_{rs}^2
	\prod_{i\neq r,s}(\lambda+\mu_i)
	\equiv0 .
\end{split}
\]
Fix \(j\in\{1,2,\cdots,N\}\). Since the numbers
\(\mu_1,\cdots,\mu_N\) are pairwise distinct, we may evaluate this
polynomial identity at \(\lambda=-\mu_j\). Then all terms in the
first sum vanish except the term with \(r=j\), and all terms in the
second sum vanish except those involving the index \(j\). This gives that
\[
\left[
(u_j')^2+(S+\mu_j)u_j^2
\right]
\prod_{i\neq j}(\mu_i-\mu_j)
+
\sum_{l\neq j}
(u_j'u_l-u_ju_l')^2
\prod_{i\neq j,l}(\mu_i-\mu_j)
=0.
\]
It then follows that
\[
(u_j')^2+(S+\mu_j)u_j^2
+
\sum_{l\neq j}
\frac{(u_j'u_l-u_ju_l')^2}{\mu_l-\mu_j}
=0.
\]
This proves \eqref{eq:algebraic-identity-function}.
\end{proof}

Applying the above two lemmas, we are now ready to prove Proposition \ref{thm:algiden}.

\begin{proof}[\normalfont \bfseries Proof of Proposition \ref{thm:algiden}]
By Theorem~\ref{thm:completeness-multiple-root}, we have
\(u\in\mathcal M_{\mathrm{multi}}\). Hence there exist
\[
a=(a_1,\cdots,a_K)\in(\mathbb R\setminus\{0\})^K\,\ \hbox{and}\,\ n^{(\alpha)}
=
(n_j^{(\alpha)})_{j\in I_\alpha}
\in\mathbb S^{k_\alpha-1},
\ \  \alpha=1,\cdots,K,
\]
such that
$
u_j(x)=n_j^{(\alpha)}V_\alpha^a(x)
$ holds for $j\in I_\alpha$.
For simplicity, write \(V_\alpha=V_\alpha^a\). Since
$
\sum_{j\in I_\alpha}\bigl(n_j^{(\alpha)}\bigr)^2=1,
$
we have
\[
\sum_{j\in I_\alpha}u_j^2=V_\alpha^2,
\ \
\sum_{j\in I_\alpha}(u_j')^2=(V_\alpha')^2.
\]
Consequently,
\[
S=\sum_{j=1}^N u_j^2
=
\sum_{\gamma=1}^K V_\gamma^2.
\]
Moreover, we have for \(\alpha\neq\beta\), \(j\in I_\alpha\), and
\(l\in I_\beta\),
\[
u_j'u_l-u_ju_l'
=
n_j^{(\alpha)}n_l^{(\beta)}
\bigl(V_\alpha'V_\beta-V_\alpha V_\beta'\bigr).
\]
We then have
\[
\begin{aligned}
	\sum_{\substack{j\in I_\alpha\\ l\in I_\beta}}
	(u_j'u_l-u_ju_l')^2
	&  =
	\left(
	\sum_{j\in I_\alpha}
	\bigl(n_j^{(\alpha)}\bigr)^2
	\right)
	\left(
	\sum_{l\in I_\beta}
	\bigl(n_l^{(\beta)}\bigr)^2
	\right)
	\bigl(V_\alpha'V_\beta-V_\alpha V_\beta'\bigr)^2  \\
	&  =
	\bigl(V_\alpha'V_\beta-V_\alpha V_\beta'\bigr)^2 .
\end{aligned}
\]

Substituting these identities into the left-hand side of
\eqref{eq:algebraic-identity}, we obtain that
\[
\mbox{LHS  of \eqref{eq:algebraic-identity}}= (V_\alpha')^2
+
\left(
\sum_{\gamma=1}^K V_\gamma^2+\mu_\alpha
\right)V_\alpha^2
+
\sum_{\beta\neq\alpha}
\frac{
	\bigl(V_\alpha'V_\beta-V_\alpha V_\beta'\bigr)^2
}{\mu_\beta-\mu_\alpha}.
\]
The reduced profile
$
V=(V_1,\cdots,V_K)
$
solves the distinct-root system
\[
V_\alpha''
+
2\left(\sum_{\gamma=1}^K V_\gamma^2\right)V_\alpha
+
\mu_\alpha V_\alpha
=0,
\ \  \alpha=1,\cdots,K.
\]
Applying Lemma~\ref{lem:algebraic-identity-function} to this
\(K\)-component distinct-root system, we then obtain that
\[
\mbox{LHS  of \eqref{eq:algebraic-identity}}= (V_\alpha')^2
+
\left(
\sum_{\gamma=1}^K V_\gamma^2+\mu_\alpha
\right)V_\alpha^2
+
\sum_{\beta\neq\alpha}
\frac{
	\bigl(V_\alpha'V_\beta-V_\alpha V_\beta'\bigr)^2
}{\mu_\beta-\mu_\alpha}
=0,
\]
and we are done.
\end{proof}

\subsection{The \(L^2\)-mass formula}
In this subsection, we discuss the exact \(L^2\)-mass formula of Theorem \ref{thm L2iden}. We remark that the pointwise algebraic identities proved in the previous subsection allow us to express the mass of each spectral block in terms of the corresponding spectral parameter.

%Integrating these identities over \(\mathbb R\) and using the decay of the classified solutions at infinity then gives the desired formulas.

\begin{proof}[\normalfont \bfseries Proof of Theorem~\ref{thm L2iden}]
Recall that
$S=\sum_{i=1}^N u_i^2$.
For each
\(\alpha=1,\cdots,K\),   Theorem~\ref{thm:completeness-multiple-root} shows that there exist
\(n^{(\alpha)}=(n_i^{(\alpha)})_{i\in I_\alpha}\in
\mathbb S^{k_\alpha-1}\) and a nontrivial function \(v_\alpha\in
H^1(\mathbb R)\cap C^2(\mathbb R)\) such that
\[
u_i=n_i^{(\alpha)}v_\alpha,\ \  i\in I_\alpha .
\]
Hence $S=\sum_{\gamma=1}^K v_\gamma^2$
and
\[
\int_{\mathbb R}\sum_{i\in I_\alpha}u_i^2\,dx
=
\int_{\mathbb R}v_\alpha^2\,dx\neq 0.
\]
It next suffices to prove that
\[
\int_{\mathbb R}v_\alpha^2\,dx=2\eta_\alpha .
\]

The reduced functions \(v_\alpha\) satisfy
\begin{equation}\label{eq:reduced-v-equation}
	v_\alpha''+(2S+\mu_\alpha)v_\alpha=0,
	\ \  \alpha=1,\cdots,K .
\end{equation}
Fix \(\alpha\)  and  define   for \(\beta\neq\alpha\),
\[
\delta_{\alpha\beta}:=\mu_\beta-\mu_\alpha,
\ \
W_{\alpha\beta}:=v_\alpha'v_\beta-v_\alpha v_\beta' .
\]
We then derive from \eqref{eq:reduced-v-equation} that
\begin{equation}\label{eq:Wab-prime}
	W_{\alpha\beta}'
	=
	\delta_{\alpha\beta}v_\alpha v_\beta .
\end{equation}
Moreover, Proposition~\ref{thm:algiden} gives the reduced identity
\begin{equation}\label{eq:reduced-alg-identity}
	(v_\alpha')^2+(S+\mu_\alpha)v_\alpha^2
	+
	\sum_{\beta\neq\alpha}
	\frac{W_{\alpha\beta}^2}{\delta_{\alpha\beta}}
	=0 .
\end{equation}
Define on the open set \(\{v_\alpha\neq0\}\),
\[
Q_\alpha
:=
\frac{v_\alpha'}{v_\alpha}
+
\sum_{\beta\neq\alpha}
\frac{W_{\alpha\beta}v_\beta}
{\delta_{\alpha\beta}v_\alpha}.
\]
We first compute \(Q_\alpha'\) on the open set \(\{v_\alpha\neq0\}\): it follows from
\eqref{eq:reduced-v-equation} and \eqref{eq:reduced-alg-identity} that
\[
\begin{aligned}
	\left(\frac{v_\alpha'}{v_\alpha}\right)'
	&=
	\frac{v_\alpha''}{v_\alpha}
	-
	\frac{(v_\alpha')^2}{v_\alpha^2}  \\
	&=
	-(2S+\mu_\alpha)
	-
	\frac{(v_\alpha')^2}{v_\alpha^2}  \\
	&=
	-S+
	\sum_{\beta\neq\alpha}
	\frac{W_{\alpha\beta}^2}
	{\delta_{\alpha\beta}v_\alpha^2}.
\end{aligned}
\]
We also derive from \eqref{eq:Wab-prime} that
\[
\begin{aligned}
	\left(
	\frac{W_{\alpha\beta}v_\beta}
	{\delta_{\alpha\beta}v_\alpha}
	\right)'
	&=
	\frac{W_{\alpha\beta}'v_\beta+W_{\alpha\beta}v_\beta'}
	{\delta_{\alpha\beta}v_\alpha}
	-
	\frac{W_{\alpha\beta}v_\beta v_\alpha'}
	{\delta_{\alpha\beta}v_\alpha^2} \\
	&=
	v_\beta^2
	+
	\frac{W_{\alpha\beta}}
	{\delta_{\alpha\beta}v_\alpha^2}
	\left(v_\alpha v_\beta'-v_\alpha'v_\beta\right) \\
	&=
	v_\beta^2
	-
	\frac{W_{\alpha\beta}^2}
	{\delta_{\alpha\beta}v_\alpha^2}.
\end{aligned}
\]
Summing over \(\beta\neq\alpha\), it then yields that
\begin{equation}\label{eq:Q-prime}
	Q_\alpha'=-v_\alpha^2
	\ \  \text{on } \{v_\alpha\neq0\}.
\end{equation}

It remains to justify that \(Q_\alpha\) can be extended across the zeros of
\(v_\alpha\). Write
\[
Q_\alpha=\frac{P_\alpha}{v_\alpha},
\ \
P_\alpha:=
v_\alpha'
+
\sum_{\beta\neq\alpha}
\frac{W_{\alpha\beta}v_\beta}{\delta_{\alpha\beta}} ,
\]
and let \(x_0\) be a zero of \(v_\alpha\). Since \(v_\alpha\) is a nontrivial
solution of the linear equation \eqref{eq:reduced-v-equation}, we have
\(v_\alpha'(x_0)\neq0\). Hence, \(x_0\) is a simple zero of \(v_\alpha\).
On the other hand, evaluating \eqref{eq:reduced-alg-identity} at \(x_0\)
and using
\[
W_{\alpha\beta}(x_0)=v_\alpha'(x_0)v_\beta(x_0),
\]
we obtain that
\[
0
=
(v_\alpha'(x_0))^2
\Big(
1+
\sum_{\beta\neq\alpha}
\frac{v_\beta(x_0)^2}{\delta_{\alpha\beta}}
\Big).
\]
Since \(v_\alpha'(x_0)\neq0\), this implies that
\[
1+
\sum_{\beta\neq\alpha}
\frac{v_\beta(x_0)^2}{\delta_{\alpha\beta}}
=0 .
\]
Consequently,
\[
P_\alpha(x_0)
=
v_\alpha'(x_0)
\left(
1+
\sum_{\beta\neq\alpha}
\frac{v_\beta(x_0)^2}{\delta_{\alpha\beta}}
\right)
=0 .
\]
Thus, both \(P_\alpha\) and \(v_\alpha\) vanish at \(x_0\), while
\(v_\alpha\) has a simple zero at \(x_0\). So \(P_\alpha/v_\alpha\) has a
finite limit at \(x_0\). We extend \(Q_\alpha\) to \(x_0\) by this limit.
Since \(P_\alpha\) and \(v_\alpha\) are \(C^1\), this extension is continuous.

Extending \(Q_\alpha\) by the above removable limits, we obtain a continuous
function on \(\mathbb R\). Since the zeros of \(v_\alpha\) are isolated,
there are only finitely many zeros in any compact interval. Hence, integrating
\eqref{eq:Q-prime} on each nodal interval and summing them, the interior boundary
terms cancel by the continuity of \(Q_\alpha\). We thus have for any
\(R_1<R_2\),
\begin{equation}\label{eq:int-v-Q}
\int_{R_1}^{R_2}v_\alpha^2\,dx
=
Q_\alpha(R_1)-Q_\alpha(R_2).
\end{equation}

Finally, we derive from Lemma~\ref{lem:L1-asymptotic} that there exist nonzero constants
\(c_\gamma^-\) and \(c_\gamma^+\) such that  for every
\(\gamma=1,\cdots,K\),
\[
v_\gamma(x)=c_\gamma^-e^{\eta_\gamma x}
+o(e^{\eta_\gamma x}),
\ \
v_\gamma'(x)=\eta_\gamma c_\gamma^-e^{\eta_\gamma x}
+o(e^{\eta_\gamma x})
\quad \text{as }\ x\to-\infty,
\]
and
\[
v_\gamma(x)=c_\gamma^+e^{-\eta_\gamma x}
+o(e^{-\eta_\gamma x}),
\ \
v_\gamma'(x)=-\eta_\gamma c_\gamma^+e^{-\eta_\gamma x}
+o(e^{-\eta_\gamma x})
\quad \text{as }\ x\to+\infty,
\]
where $\eta_\gamma=\sqrt{|\mu_\gamma|}$ for $\gamma=1,2,\cdots,K$.
The right-hand asymptotics follow by applying
Lemma~\ref{lem:L1-asymptotic} to the reflected function
\(\widetilde v_\gamma(x):=v_\gamma(-x)\). The corresponding leading
coefficient cannot vanish, since otherwise Lemma~\ref{lem:L1-asymptotic} would
imply that \(v_\gamma\) vanishes on a right half-line, and hence
\(v_\gamma\equiv0\) by the uniqueness of ODE.
We hence have
\[
\frac{v_\alpha'}{v_\alpha}\to\eta_\alpha
\quad\text{as }\ x\to-\infty,
\ \
\frac{v_\alpha'}{v_\alpha}\to-\eta_\alpha
\quad\text{as }\ x\to+\infty.
\]
Moreover, for every \(\beta\neq\alpha\),
\[
\frac{W_{\alpha\beta}v_\beta}{v_\alpha}
=O(e^{2\eta_\beta x})
\quad\text{as }\ x\to-\infty,
\]
and
\[
\frac{W_{\alpha\beta}v_\beta}{v_\alpha}
=O(e^{-2\eta_\beta x})
\quad\text{as }\ x\to+\infty.
\]
These yield that
\[
Q_\alpha(x)\to\eta_\alpha
\quad\text{as }\ x\to-\infty,
\ \
Q_\alpha(x)\to-\eta_\alpha
\quad\text{as } \ x\to+\infty.
\]
Setting \(R_1\to-\infty\) and \(R_2\to+\infty\) in
\eqref{eq:int-v-Q}, we conclude that
\[
\int_{\mathbb R}v_\alpha^2\,dx
=
\eta_\alpha-(-\eta_\alpha)
=
2\eta_\alpha=2\sqrt{|\mu_\alpha|}.
\]
Consequently, we have for every $\alpha=1,2,\cdots,K$,
\[
\int_{\mathbb R}\sum_{i\in I_\alpha}u_i^2\,dx
=
\int_{\mathbb R}v_\alpha^2\,dx
=
2\eta_\alpha=2\sqrt{|\mu_\alpha|}.
\]
This therefore completes the proof of Theorem~\ref{thm L2iden}.
\end{proof}

\paragraph{Declaration of AI-assisted tools.}
During the preparation of this manuscript, the authors used OpenAI models as auxiliary tools for exploratory discussions, preliminary proof drafting and language editing. The
mathematical validation, the final proof choices, and the final text are the responsibility
of the human authors.


\begin{thebibliography}{99}
\bibitem{A}
A. Ambrosetti and E. Colorado,
\emph{Standing waves of some coupled nonlinear Schr\"{o}dinger equations},
J. Lond. Math. Soc. \textbf{75} (2007), 67--82.

\bibitem{BS}
T. Bartsch and N. Soave,
\emph{A natural constraint approach to normalized solutions of nonlinear Schr\"{o}dinger equations and systems},
J. Funct. Anal. \textbf{272} (2017), 4998--5037.

\bibitem{BZZ}
T. Bartsch, X. Zhong and W. Zou,
\emph{Normalized solutions for a coupled Schr\"{o}dinger system},
Math. Ann. \textbf{380} (2021), 1713--1740.

\bibitem{CS}
M. Clapp and A. Szulkin,
\emph{Normalized solutions to a non-variational Schr\"{o}dinger system},
Topol. Methods Nonlinear Anal. \textbf{61} (2023), 445--464.

\bibitem{CoddingtonLevinson1955}
E. A. Coddington and N. Levinson,
{Theory of Ordinary Differential Equations},
McGraw--Hill, New York, 1955.

\bibitem{Eastham1989}
M. S. P. Eastham,
\emph{The asymptotic solution of linear differential systems:
	applications of the Levinson theorem},
Clarendon Press, Oxford, 1989.

\bibitem{Lewin3}
R. Frank, D. Gontier and M. Lewin,
\emph{Optimizers for the finite-rank Lieb--Thirring inequality},
Amer. J. Math. \textbf{147} (2025), 503--560.

\bibitem{Lewin}
R. Frank, D. Gontier and M. Lewin,
\emph{The nonlinear Schr\"{o}dinger equation for orthonormal functions II:
	application to Lieb--Thirring inequalities},
Commun. Math. Phys. \textbf{384} (2021), 1783--1828.

\bibitem{FI}
B. D. Fried and Y. H. Ichikawa,
\emph{On the nonlinear Schr\"{o}dinger equation for Langmuir waves},
J. Phys. Soc. Japan. \textbf{34} (1973), 1073--1082.

\bibitem{GLN}
D. Gontier, M. Lewin and F. Q. Nazar,
\emph{The nonlinear Schr\"{o}dinger equation for orthonormal functions: existence of ground states},
Arch. Ration. Mech. Anal. \textbf{240} (2021), 1203--1254.


\bibitem{GLW2026}
Y. Guo, Y. Luo and J. Wei,
\emph{Classification and nondegeneracy of cubic nonlinear Schr\"{o}dinger systems in $\mathbb R$},
Anal. PDE (2026),  to appear.

\bibitem{H}
R. Hirota,
\emph{Direct method in soliton theory},
 Topics in Current Physics Vol. 17,
Springer, Berlin, 1980, pp. 157--176.

\bibitem{HornJohnson2013}
R. A. Horn and C. R. Johnson,
\emph{Matrix Analysis}, 2nd ed.,
Cambridge University Press, Cambridge, 2013.

\bibitem{HLT}
D. Hundertmark, E. H. Lieb and L. E. Thomas,
\emph{A sharp bound for an eigenvalue moment of the one-dimensional Schr\"{o}dinger operator},
Adv. Theor. Math. Phys. \textbf{2} (1998), 719--731.

\bibitem{Levinson1948}
N. Levinson,
\emph{The asymptotic nature of solutions of linear systems of differential equations},
Duke Math. J. \textbf{15} (1948), 111--126.

\bibitem{LS}
E. H. Lieb, R. Seiringer, J. P. Solovej and J. Yngvason,
\emph{The mathematics of the Bose gas and its condensation},
Oberwolfach Seminars, Birkh\"auser, Basel, 2005.

\bibitem{LT}
E. H. Lieb and W. E. Thirring,
\emph{Inequalities for the moments of the eigenvalues of the Schr\"{o}dinger Hamiltonian and their relation to Sobolev inequalities},
Studies in Mathematical Physics, Princeton University Press,
Princeton, 1976, pp. 269--303.

\bibitem{LW}
T. C. Lin and J. C. Wei,
\emph{Ground state of $N$ coupled nonlinear Schr\"{o}dinger equations in $\mathbb R^n$, $n\leq 3$},
Commun. Math. Phys. \textbf{255} (2005), 629--653.

\bibitem{LWY}
C. S. Lin, J. C. Wei and D. Ye,
\emph{Classification and nondegeneracy of $SU(n+1)$ Toda system with singular sources},
Invent. Math. \textbf{190} (2012), 169--207.

\bibitem{M}
B. Malomed,
\emph{Nonlinear Schr\"{o}dinger equations},
Encyclopedia of Nonlinear Science,
Routledge, 2005, pp. 639--642.

\bibitem{RL}
R. Radhakrishnan and M. Lakshmanan,
\emph{Bright and dark soliton solutions to coupled nonlinear Schr\"{o}dinger equations},
J. Phys. A \textbf{28} (1995), 2683--2692.

\bibitem{RSL}
R. Ramakrishnan, S. Stalin and M. Lakshmanan,
\emph{Multihumped nondegenerate fundamental bright solitons in $N$-coupled nonlinear Schr\"{o}dinger system},
J. Phys. A: Math. Theor. \textbf{54} (2021), 14LT01.

\bibitem{Rama2}
A. O. Smirnov, E. A. Frolov and V. S. Gerdjikov,
\emph{Spectral curves for the multi-phase solutions of Manakov system},
IOP Conf. Ser.: Mater. Sci. Eng. \textbf{862} (2020), 052041.

\bibitem{W}
T. Weidl,
\emph{On the Lieb--Thirring constants $L_{\gamma,1}$ for $\gamma \geq 1/2$},
Comm. Math. Phys. \textbf{178} (1996), 135--146.

\bibitem{WZZ}
J. C. Wei, X. X. Zhong and W. M. Zou,
\emph{On Sirakov's open problem and related topics},
Ann. Sc. Norm. Super. Pisa Cl. Sci. \textbf{23} (2022), no. 2, 959--992.

\bibitem{WW}
J. C. Wei and Y. Z. Wu,
\emph{Ground states of nonlinear Schr\"{o}dinger systems with mixed couplings},
J. Math. Pures Appl. \textbf{141} (2020), 50--88.

\bibitem{ZS}
V. E. Zakharov and A. B. Shabat,
\emph{Exact theory of two-dimensional self-focusing and one-dimensional self-modulation of waves in nonlinear media},
Sov. Phys. JETP \textbf{34} (1972), 62--69.
	




\end{thebibliography}
\end{document}